\documentclass{amsart}

\usepackage{amsmath}
\usepackage{amssymb}
\usepackage{verbatim}
\usepackage{enumerate}
\usepackage[all]{xy}
\usepackage[mathscr]{eucal}

\usepackage{amsthm}

\usepackage{amscd}
\usepackage{amsfonts}
\usepackage{latexsym}

\usepackage{amscd}

\usepackage{graphicx}

\def\bd#1{\mbox{\boldmath${#1}$}}

\newtheorem{theorem}{Theorem}[section]
\newtheorem{lemma}[theorem]{Lemma}
\newtheorem{proposition}[theorem]{Proposition}
\newtheorem{corollary}[theorem]{Corollary} 
\theoremstyle{definition}  
\newtheorem{definition}[theorem]{Definition}
\newtheorem{example}[theorem]{Example}
\newtheorem{conjecture}[theorem]{Conjecture}  

\newtheorem{remark}[theorem]{Remark}

\newcommand{\id}{\text{id}}

\newcommand{\End}{\text{End}}

\newcommand{\Hom}{\operatorname{Hom}} 
 
\def\uRep{\underline{\operatorname{Re}}\!\operatorname{p}}

\newcommand{\Lift}{\operatorname{Lift}}
\newcommand{\Rep}{\operatorname{Rep}}
\newcommand{\Inc}{\operatorname{Inc}}
\newcommand{\Proj}{\operatorname{Proj}}
\newcommand{\uRes}{\underline{\operatorname{Res}}}

\newcommand{\tr}{\text{tr}\,}

\newcommand{\Z}{\mathbb{Z}}
\newcommand{\cat}{\mathcal}
\newcommand{\ep}{\varepsilon}
\newcommand{\Dox}{\rotatebox{45}{$\Box$} }

\begin{document}

\title{On blocks of Deligne's category $\uRep(S_t)$.}

\author{Jonathan Comes}
\address{J.C.: Technische Universit\"{a}t M\"{u}nchen,
Zentrum Mathematik, 85748 Garching, Germany}
\email{comes@ma.tum.de}

\author{Victor Ostrik}
\address{V.O.: Department of Mathematics,
University of Oregon, Eugene, OR 97403, USA}
\email{vostrik@math.uoregon.edu}

\begin{abstract}
We describe blocks in Deligne's category $\uRep(S_t)$.
\end{abstract}

\date{\today}
\maketitle  

\section{Introduction}

Let $F$ be a field of characteristic zero and $t\in F$. 
Recently P.~Deligne introduced the tensor category $\uRep(S_t)$ which in a certain precise
sense interpolates the categories $\Rep(S_d)$ of $F-$linear representations  of the symmetric 
groups $S_d$, see \cite{Del07}. In general, the category $\uRep(S_t)$ is additive but not abelian; 
however it is proved in
\cite[Th\'eor\`eme 2.18]{Del07} that for $t$ which is not a nonnegative integer the category $\uRep(S_t)$
is semisimple (and hence abelian). 

The goal of this paper is to get a better understanding of the additive category $\uRep(S_t)$ in the case
when it is not semisimple. To this end we employ the usual strategy from representation theory:
we split our category into {\em blocks} (see \S \ref{blocksdef}) and analyze the blocks separately. 
In order to split a category into blocks, it is a standard technique to study
the endomorphism ring of the identity functor in this category (recall that for an algebra $A$
the endomorphism ring of the identity functor of the category of $A-$modules is precisely the
center of $A$). 
Thus we start by constructing many nontrivial endomorphisms of the identity functor of
the category $\uRep(S_t)$ (see \S \ref{endId}) and studying how these endomorphisms act on 
the indecomposable objects of the category $\uRep(S_t)$. In \S \ref{blocks} we show that this is
sufficient for describing all the blocks in the category $\uRep(S_t)$; in particular
we give a new proof of \cite[Th\'eor\`eme 2.18]{Del07}, see Corollary \ref{ss}. In general we find
that the category $\uRep(S_t)$ contains infinitely many {\em trivial} blocks (they are equivalent
to the category of finite dimensional $F-$vector spaces as additive categories) and finitely many
non-semisimple blocks. In \S \ref{blockquiver} we show that all these non-semisimple blocks are
equivalent to each other as additive categories and give a description of this category via
quiver with relations.

One of the very interesting results of \cite{Del07} is a construction of an {\em abelian} tensor
category $\uRep^{ab}(S_t)$ and fully faithful tensor functor $\uRep(S_t)\to \uRep^{ab}(S_t)$, 
see \cite[Proposition 8.19]{Del07}. In a subsequent publication we plan to apply the
results of the present paper to the study of $\uRep^{ab}(S_t)$; in particular we plan to prove
its universal property as in \cite[Conjecture 8.21]{Del07}. In another direction, F.~Knop
significantly generalized Deligne's construction in \cite{MR2349713}; we hope that our
approach will be useful in the study of tensor categories from {\em loc. cit.}

The paper is organized as follows. Sections 2 and 3 are mostly expository. In Section 2 
we give a detailed construction of category $\uRep(S_t)$ with emphasis on motivation. 
In Section 3 we give a classification of indecomposable objects in the category $\uRep(S_t)$
and discuss the relation of $\uRep(S_d)$ and $\Rep(S_d)$ (this is closely related with 
\cite[\S 5-6]{Del07}). In Section 4 we construct many
endomorphisms of the identity functor of the category $\uRep(S_t)$; this is the main technical
tool of this paper. In Section 5 we give a complete classification of the blocks in the category
$\uRep(S_t)$. Finally in Section 6 we give a quiver description of the non-semisimple blocks
in the category $\uRep(S_t)$.

It is a great pleasure to thank here Alexander Kleshchev who suggested we study
the category $\uRep(S_t)$ and taught us about representation theory of the symmetric group.
V.O. also acknowledges inspiring conversations with Pierre Deligne and Pavel Etingof.   J.C. is  
indebted to Anne Henke and Friederike Stoll for their comments on improving the paper.
Both authors are grateful to the referee for carefully reading the paper and providing useful suggestions.
Part of the work on this paper was done when the second author enjoyed hospitality
of the Institute for Advanced Study; he is happy to thank this institution for 
the excellent research conditions.  
This work was  partially supported by the NSF grant DMS-0602263.

\subsection{Terminology and notation conventions}\label{notation}  Let $F$ be a field.  A \emph{tensor category} is an $F$-linear\footnote{Notice that we do not require our tensor categories to be abelian, or even additive.} category $\cat{T}$ along with an $F$-linear bifunctor $\otimes:\cat{T}\times\cat{T}\to\cat{T}$ and associativity, commutativity, and unit constraints satisfying the triangle, pentagon, and hexagon axioms (see \cite{MR1797619}).  Moreover, we require our tensor categories to be rigid with $\End_\cat{T}({\bf 1})=F$.  

A \emph{Young diagram} is an infinite tuple of nonnegative integers  $\lambda=(\lambda_1, \lambda_2,\ldots)$ with $\lambda_i\geq\lambda_{i+1}$ for all $i>0$ such that $\lambda_k=0$ for some $k>0$.  As usual, we identify a Young diagram with an array of boxes: 
$$\includegraphics{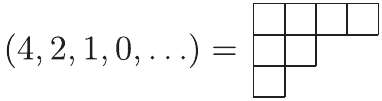}$$  Let $|\lambda|$ denote the number of boxes in $\lambda$. 
Given Young diagrams $\lambda=(\lambda_1, \lambda_2,\ldots)$ and $\lambda'=(\lambda_1', \lambda_2',\ldots)$ write $\lambda\subset\lambda'$ if $\lambda_i\leq\lambda_i'$ for all $i>0$.  If $|\lambda|=|\lambda'|$ we write $\lambda\prec\lambda'$ if $\sum_{i=1}^m\lambda_i\leq\sum_{i=1}^m\lambda_i'$ for all $m>0$.  Occasionally it will be useful to give a Young diagram by its \emph{multiplicities}.  We will write $(l_1^{m_1},\ldots, l_r^{m_r})$ to denote the Young diagram with $m_i$ rows of length $l_i$ for each $i=1,\ldots, r$.  If $F$ has characteristic zero and $d$ is a nonnegative integer, the simple modules of the symmetric group $S_d$ are labelled by Young diagrams $\lambda$ with $|\lambda|=d$ (see e.g. \cite[4.2]{MR1153249}).  Let $L_\lambda$ denote the simple $S_d$-module corresponding to $\lambda$.  For example, $L_{(d,0,\ldots)}$ is the trivial representation of $S_d$.
Finally, for arbitrary $t\in F$, the \emph{$t$-completion of $\lambda$} is $\lambda(t):=(t-|\lambda|, \lambda_1, \lambda_2,\ldots)$.

\section{The tensor category $\uRep(S_t)$}

\subsection{Motivation}\label{motivation}

Let $d$ be a nonnegative integer, and let $F$ be a field of characteristic zero.  
Let us consider the tensor category $\Rep(S_d; F)$ of finite dimensional representations over $ F$ of the symmetric group $S_d$.  Note that we take $S_0$ to be the trivial group whose one element is the identity permutation of the empty set.  Let $V_d$ denote the natural $d$-dimensional representation of $S_d$ with basis $\{v_1,\ldots, v_d\}$, so that $S_d$ acts by permuting the basis elements ($V_0$ is taken to be 0).    Setting $V_d^{\otimes 0}=F$ for all $d\geq 0$, we have the following well-known result:

\begin{proposition}\label{full} Any irreducible representation of $S_d$ is a direct summand of $V_d^{\otimes n}$ for some nonnegative integer $n$.
\end{proposition}



As a consequence of Proposition \ref{full}, one way to understand $ \Rep(S_d; F)$ is to study objects of the form $V^{\otimes n}_d$ and morphisms between those objects.  We will now use set partitions to construct some such morphisms.  Our notation will be similar to that of \cite{MR2143201}.

By a \emph{partition} $\pi$ of a finite set $S$ we mean a collection $\pi_1,\ldots,\pi_n$
of disjoint subsets of $S$ with $S=\bigcup_i \pi_i$. The sets
$\pi_i$ will be called \emph{parts} of $\pi$.  Given a partition $\pi$ of $\{1,\ldots, n, 1',\ldots, m'\}$, a \emph{partition diagram of $\pi$} is any graph with vertices labelled $\{1,\ldots, n, 1',\ldots, m'\}$ whose connected components partition the vertices into the parts of $\pi$.  We will always draw partition diagrams using the following convention:
\begin{itemize}
\item Vertices $1,\ldots, n$ ~(resp. $1',\ldots, m'$) are aligned horizontally and increasing \\
from left to right with $i$ directly above $i'$.
\item Edges lie entirely below the vertices labelled $1,\ldots, n$ and above the vertices \\ labelled $1',\ldots, m'$.
\end{itemize}
We will abuse notation by writing $\pi$ for both the partition and the partition diagram.

\begin{example}\label{pd} The graph
$$\includegraphics{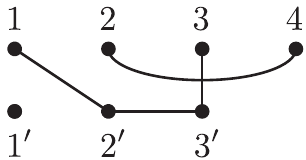}$$
 is a partition diagram for the partition
$\{\{1,3,2',3'\},\{2,4\},\{1'\}\}$.  Note a
partition diagram representing this partition is not unique, but its connected components are.$\hfill\Dox$
\end{example}

To each partition of the set $\{1,\ldots, n, 1',\ldots, m'\}$ we will associate a linear map $V_d^{\otimes n}\to V^{\otimes m}_d$.  Before doing so, we introduce some notation.

\begin{itemize}\item Let $P_{n,m}$ denote the set of all partitions of $\{1,\ldots, n, 1',\ldots, m'\}$, let $P_{n,0}$ denote the set of all partitions of $\{1,\ldots, n\}$, let $P_{0,m}$ denote the set of all partitions of $\{1',\ldots, m'\}$, and let $P_{0,0}$ denote the set consisting of the empty partition.  Finally,  let $ F P_{n, m}$ denote the $ F$-vector space with basis $P_{n,m}$.

\item For nonnegative integers $n$ and $d$, let $[n, d]$ denote the set of all functions from $\{j~|~1\leq j\leq n\}$ to $\{j~|~1\leq j\leq d\}$.  In particular, $[0, d]=\{\varnothing\}$ for all $d$, and $[n, 0]=\varnothing$ for all $n\not=0$.   Given ${\bd{i}}\in[n, d]$ and $j\in\{j~|~1\leq j\leq n\}$, write $i_j$ for the image of $j$ under ${\bd{i}}$.

\item For ${\bd{i}}\in[n, d]$ and ${\bd{i}}'\in[m, d]$, the \emph{$({\bd{i}}, {\bd{i}}')$-coloring} of a partition $\pi\in P_{n, m}$ is obtained by coloring the vertices of $\pi$ labelled $j$ (resp. $j'$) by the integer $i_j$ (resp. $i'_j$).  
Saying an $({\bd{i}}, {\bd{i}}')$-coloring of $\pi$ is \emph{good} means vertices are colored the same whenever they are in the same connected component of $\pi$.  Saying an $({\bd{i}}, {\bd{i}}')$-coloring of $\pi$ is \emph{perfect} means vertices are colored the same if and only if they are in the same connected component of $\pi$. 

\item For $n\not=0$ and ${\bd{i}}\in[n, d]$, set $v_{\bd{i}}:=v_{i_1}\otimes\cdots\otimes v_{i_n}\in V_d^{\otimes n}$.   Set $v_\varnothing:=1\in V^{\otimes0}_d$ where $\varnothing$ is the unique element of $[0, d]$.  

\end{itemize}

We are now ready to associate partitions in $P_{n,m}$ with linear maps $V_d^{\otimes n}\to V^{\otimes m}_d$.

\begin{definition}\label{f} For nonnegative integers $n, m$, and $d$, define the $F$-linear map $f:FP_{n, m}\to\Hom_{S_d}(V_d^{\otimes n}, V_d^{\otimes m})$ by setting  $$f(x)(v_{\bd{i}})=\sum_{{\bd{i}}'\in[m, d]}f(x)^{\bd{i}}_{ {\bd{i}}'}v_{{\bd{i}}'}\qquad (x\in FP_{n, m}, \bd{i}\in[n, d])$$ where 
$$f(\pi)^{\bd{i}}_{{\bd{i}}'}:=\left\{\begin{array}{ll}
1, & \text{if the $({\bd{i}}, {\bd{i}}')$-coloring of $\pi$ is good},\\
0, & \text{otherwise}.
\end{array}\right.\qquad (\pi\in P_{n,m})$$
\end{definition}
Indeed, $f(x)$ commutes with the action of $S_d$ which merely permutes the colors.

\begin{example} (1) If $\varnothing\in P_{0,0}$ denotes the empty partition, then $f(\varnothing):F\to F$ is the identity map.

(2) Assume $d>0$ and let $\pi$ be the partition with partition diagram $$\includegraphics{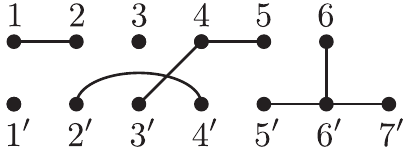}$$
Then $f(\pi):V_d^{\otimes 6}\to V_d^{\otimes 7}$ is given by $$f(\pi)(v_{\bd{i}})=\left\{\begin{array}{ll}
\sum\limits_{1\leq k, j\leq d} v_k\otimes v_j\otimes v_{i_4}\otimes v_j\otimes v_{i_6}\otimes v_{i_6}\otimes v_{i_6}, & \text{if $i_1=i_2$ and $i_4=i_5$},\\
0, & \text{otherwise}.
\end{array}\right.$$ for ${\bd{i}}\in[6,d]$.

(3) Assume $d>0$ and let $\pi$ be the partition with partition diagram $$\includegraphics{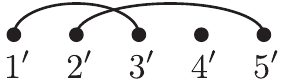}$$ Then $f(\pi):F\to V_d^{\otimes 5}$ is given by $$f(\pi)(1)=\sum_{1\leq i, j, k\leq d}v_i\otimes v_j\otimes v_i\otimes v_k\otimes v_j.$$
$\hfill\Dox$
\end{example}

Next we wish to show that $f$ is surjective.  A proof of this fact when $n=m$ can be found in \cite[Theorem 3.6]{MR2143201}.  Their proof extends to the case $n\not=m$ without difficulty.  For completeness we will give a modified version of their proof.   Before doing so, we must introduce a bit more notation.

\begin{itemize} \item Let $\leq$ be the partial order on $P_{n,m}$ defined by
$\pi\leq\mu$ whenever the partition $\mu$ is coarser than the
partition $\pi$ (i.e. $r$ and $s$ are in the same part of $\mu$
whenever they are in the same part of $\pi$ for each pair $r, s\in\{1,\ldots, n, 1',\ldots, m'\}$).

\item Define the basis $\{x_\pi~|~\pi\in P_{n,m}\}$ of $F P_{n,m}$ inductively by setting \begin{equation}\label{xpi} x_\pi:=\pi-\sum\limits_{\mu\gneqq \pi} x_\mu.\end{equation}
\end{itemize}

\begin{example} (1) $x_\pi=\pi$ when $\pi$ is any partition consisting of one part.

(2) Let $\pi\in P_{4,3}$ be the partition with the partition diagram given in Example \ref{pd}.  Then $x_\pi=\pi-\mu_1-\mu_2-\mu_3+2\mu_4$ where $$\includegraphics{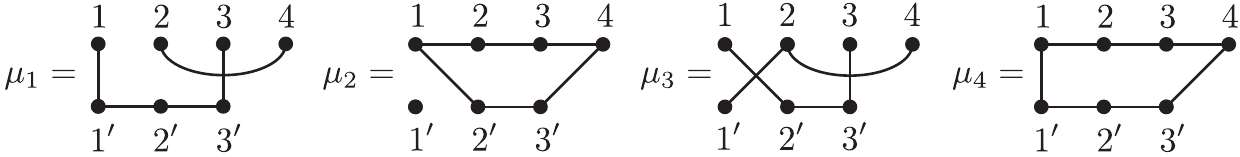}$$ $\hfill\Dox$

\end{example}

We are now ready to prove the following:

\begin{theorem}\label{sw}  For integers $n, m, d\geq 0$, the map $f: F P_{n,m}\to\Hom_{S_d}(V_d^{\otimes n}, V_d^{\otimes m})$ (see Definition \ref{f}) has the following properties:
\begin{enumerate}
\item[(1)] $f$ is surjective.

\item[(2)] $\ker(f)=\text{Span}_ F\{x_\pi~|~\pi\text{ has more than $d$ parts}\}$.  
\end{enumerate} 
In particular, $f$ is an isomorphism of $ F$-vector spaces whenever $d\geq n+m$.
\end{theorem}

\begin{proof} (compare with \cite[proof of Theorem 3.6(a)]{MR2143201})  If $d=0$ then the theorem is certainly true.  Now assume $d>0$.  For $g\in\Hom_{S_d}(V_d^{\otimes n}, V_d^{\otimes m})$ write $$g(v_{\bd{i}}):=\sum_{{\bd{i}}'\in[m, d]}g^{\bd{i}}_{{\bd{i}}'}v_{{\bd{i}}'}\qquad({\bd{i}}\in[n, d])$$  Since $g$ commutes with the action of $S_d$, the matrix entries, $g^{\bd{i}}_{{\bd{i}}'}$, are constant on the $S_d$-orbits of the matrix coordinates $\{({\bd{i}}, {\bd{i}}')\}_{{\bd{i}}\in[n,d], {\bd{i}}'\in[m,d]}$.  Each $S_d$-orbit of $\{({\bd{i}}, {\bd{i}}')\}_{{\bd{i}}\in[n,d], {\bd{i}}'\in[m,d]}$ corresponds to a partition $\pi\in P_{n,m}$ as follows: $({\bd{i}}, {\bd{i}}')$ is in the orbit corresponding to $\pi$ if and only if the $({\bd{i}}, {\bd{i}}')$-coloring of $\pi$ is perfect.  A straightforward induction argument shows 
 \begin{equation}\label{fx}
f(x_\pi)^{\bd{i}}_{{\bd{i}}'}=\left\{\begin{array}{ll}
1, & \text{if the $({\bd{i}}, {\bd{i}}')$-coloring of $\pi$ is perfect},\\
0, & \text{otherwise}.
\end{array}\right.\quad({\bd{i}}\in[n,d], {\bd{i}}'\in[m,d])
\end{equation}
Thus $g$ is an $ F$-linear combination of the $f(x_\pi)$'s.  This proves part (1).

To prove part (2) notice that for $\pi\in P_{n,m}$, there exist ${\bd{i}}\in[n, d]$ and ${\bd{i}}'\in[m, d]$ such that the $({\bd{i}}, {\bd{i}}')$-coloring of $\pi$ is perfect if and only if $\pi$ has at most $d$ parts.  Hence, by (\ref{fx}), $f(x_\pi)$ is the zero map if and only if $\pi$ has more than $d$ parts.  Part (2) now follows since $\{x_\pi~|~\pi\in P_{n,m}\}$ is a basis for $ F P_{n,m}$.
\end{proof}

We conclude our investigation of morphisms of the form $f(\pi):V_d^{\otimes n}\to V_d^{\otimes m}$ by studying the composition of such morphisms.  First we require the following:

\begin{definition}\label{star} Given partition diagrams $\pi\in P_{n,m}$ and $\mu\in P_{m,l}$, construct a new diagram $\mu\star\pi$ by identifying the vertices $1',\ldots, m'$ of $\pi$ with the vertices $1,\ldots, m$ of $\mu$ and renaming them $1'',\ldots,m''$ as illustrated below.
$$\includegraphics{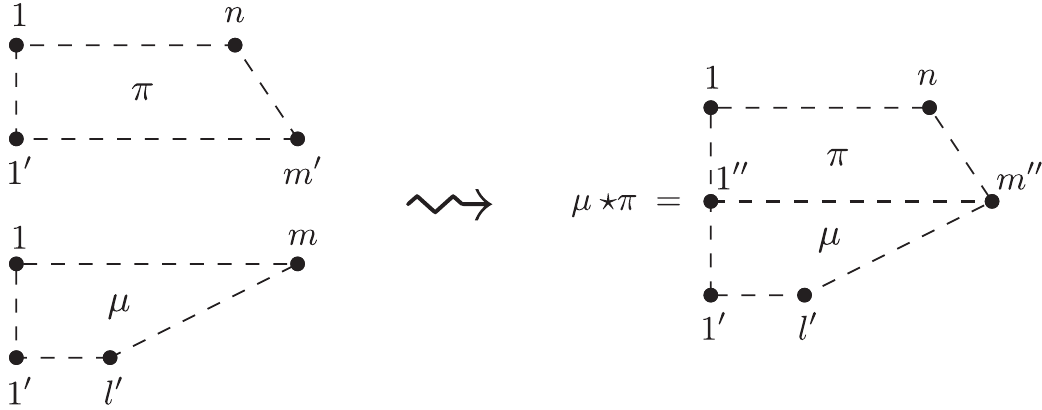}$$ 
Now let $\ell(\mu, \pi)$ denote the number of connected components of $\mu\star\pi$ whose vertices are not among $1,\ldots, n, 1', \ldots, l'$.  Finally, let $\mu\cdot\pi\in P_{n,l}$ be the partition obtained by restricting $\mu\star\pi$ to $\{1,\ldots, n, 1',\ldots, l'\}$ (i.e. $r$ and $s$ are in the same part of $\mu\cdot\pi$ if and only if $r$ and $s$ are in the same part of $\mu\star\pi$).
\end{definition}

We are now ready to state

\begin{proposition}\label{comp} $f(\mu)f(\pi)=d^{\ell(\mu, \pi)}f(\mu\cdot\pi)$ for any $\pi\in P_{n,m}$, $\mu\in P_{m,l}$.
\end{proposition}

Before we give a proof of Proposition \ref{comp} let us consider an example.

\begin{example} In this example we will verify Proposition \ref{comp} when $$\includegraphics{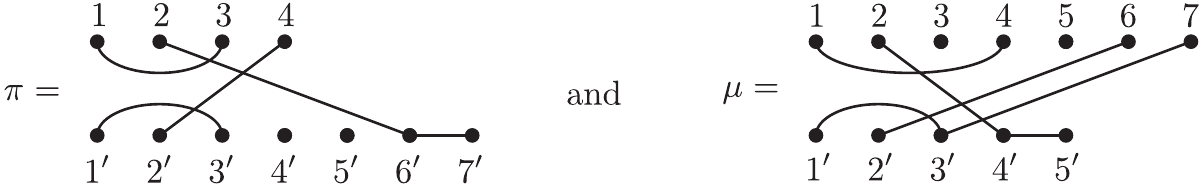}$$
On the one hand,
$$\includegraphics{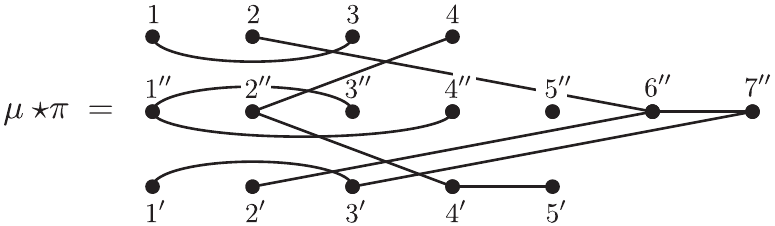}$$ Thus $\ell(\mu,\pi)=2$ and $$\includegraphics{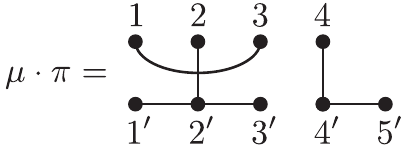}$$
Therefore we have $d^{\ell(\mu, \pi)}f(\mu\cdot\pi)(v_{i_1}\otimes v_{i_2}\otimes v_{i_3}\otimes v_{i_4})=d^2\delta_{i_1, i_3}v_{i_2}\otimes v_{i_2}\otimes v_{i_2}\otimes v_{i_4}\otimes v_{i_4}$ where $\delta_{i,j}$ is the Kronecker delta function.

On the other hand, 
$$\begin{array}{rl}
f(\mu) f(\pi)(v_{\bd{i}}) & =f(\mu)\left(\delta_{i_1,i_3}\sum\limits_{1\leq j, k, l\leq d}v_{j}\otimes v_{i_4}\otimes v_{j}\otimes v_{k}\otimes v_{l}\otimes v_{i_2}\otimes v_{i_2}\right)\\
& =\delta_{i_1,i_3}\sum\limits_{1\leq j, k, l\leq d}f(\mu)\left(v_{j}\otimes v_{i_4}\otimes v_{j}\otimes v_{k}\otimes v_{l}\otimes v_{i_2}\otimes v_{i_2}\right)\\
& =\delta_{i_1,i_3}\sum\limits_{1\leq j, k, l\leq d}\delta_{j,k}~v_{i_2}\otimes v_{i_2}\otimes v_{i_2}\otimes v_{i_4}\otimes v_{i_4} \\
& =\delta_{i_1,i_3}\sum\limits_{1\leq j, l\leq d}v_{i_2}\otimes v_{i_2}\otimes v_{i_2}\otimes v_{i_4}\otimes v_{i_4}   \\
& =d^2\delta_{i_1,i_3}~v_{i_2}\otimes v_{i_2}\otimes v_{i_2}\otimes v_{i_4}\otimes v_{i_4}. \\
\end{array}$$
 for any ${\bd{i}}\in[4, d]$, as desired.$\hfill\Dox$
\end{example}

Now to show Proposition \ref{comp} holds in general.

\medskip\noindent{\bf Proof of Proposition \ref{comp}.}
 Suppose $\pi\in P_{n,m}$ and $\mu\in P_{m, l}$ for some integers $n, m, l\geq 0$.  By Definition \ref{f} the matrix coordinates of $f(\mu) f(\pi):V_d^{\otimes n}\to V_d^{\otimes l}$ are given by 
\begin{equation}\label{comp1}(f(\mu) f(\pi))_{{\bd{i}}'}^{\bd{i}}=\sum_{{\bd{i}}''\in[m,d]}f(\mu)_{{\bd{i}}'}^{{\bd{i}}''}f(\pi)_{{\bd{i}}''}^{{\bd{i}}}.\end{equation}
Hence $(f(\mu) f(\pi))_{{\bd{i}}'}^{\bd{i}}$ is the number of ${\bd{i}}''\in[m,d]$ such that the $({\bd{i}}, {\bd{i}}'')$-coloring of $\pi$ and the $({\bd{i}}'', {\bd{i}}')$-coloring of $\mu$ are simultaneously good.  These are exactly the ${\bd{i}}''\in[m,d]$ such that coloring the vertices $j, j', j''$ of $\mu\star\pi$ with integers $i_j, i'_j, i''_j$ respectively, gives a good coloring of $\mu\star\pi$.  Any good coloring of $\mu\star\pi$ gives rise to a good coloring of $\mu\cdot\pi$.  Clearly any good coloring of $\mu\cdot\pi$ arises in this way.  Moreover, two good colorings of $\mu\star\pi$ give the same good coloring of $\mu\cdot\pi$ if and only if the two colorings of $\mu\star\pi$ differ only at connected components whose vertices are not among $1,\ldots, n, 1', \ldots, l'$.  Since there are $d$ choices of color for each component, we see the number of ${\bd{i}}''\in[m, d]$ such that the $({\bd{i}}, {\bd{i}}'')$-coloring of $\pi$ and the $({\bd{i}}'', {\bd{i}}')$-coloring of $\mu$ are simultaneously good is $d^{\ell(\mu, \pi)}f(\mu\cdot\pi)_{{\bd{i}}'}^{{\bd{i}}}$.  The result follows.
$\hfill\Box$\medskip

\begin{remark} From Proposition \ref{comp} the structure constants of the composition $f(\mu)f(\pi)$ are polynomial in the integer $d$.  We will exploit this fact in section \ref{defRep} when we define the category $\uRep(S_t; F)$ which ``interpolates" the category $\Rep(S_t; F)$ for nonnegative integer $t$, but is defined for arbitrary $t\in F$.  

\end{remark}


\subsection{Definition of $\uRep(S_t; F)$}\label{defRep} 

In this section we define Deligne's tensor category $\uRep(S_t; F)$ for arbitrary $t\in F$
following \cite[\S 8]{Del07} (so our definition is different from the one given in \cite[\S 2]{Del07}). To construct $\uRep(S_t; F)$ we will first use partitions to construct the smaller category $\uRep_0(S_t; F)$.  We then obtain $\uRep(S_t; F)$  from $\uRep_0(S_t; F)$ using the process of Karoubification.  

As in the previous section, assume $d$ is a nonnegative integer and that $F$ is a field of characteristic zero.  Let $\Rep_0(S_d; F)$ denote the full subcategory of $\Rep(S_d; F)$ whose objects are of the form $V_d^{\otimes n}$ for $n\geq 0$.  Clearly the objects in $ \Rep_0(S_d; F)$ are indexed by nonnegative integers, and by Theorem \ref{sw} the morphisms in $ \Rep_0(S_d; F)$ are given (albeit not uniquely) by $F$-linear combinations of maps $f(\pi)$ indexed by set partitions.  Moreover, the structure constants of compositions of the $f(\pi)$'s are polynomials in $d$ (see Proposition \ref{comp}).  Using this data, we now define a tensor category similar to $ \Rep_0(S_d; F)$ replacing the integer $d$ with an arbitrary element of $F$.

Let $t\in F$. 

\begin{definition}  The category $\uRep_0(S_t; F)$ has 

Objects: $[n]$ for each $n\in\Z_{\geq 0}$.

Morphisms:  $\Hom_{ \uRep_0(S_t; F)}([n], [m]):=FP_{n,m}$.  

Composition: $FP_{m,l}\times FP_{n,m}\to FP_{n,l}$ is defined to be the bilinear map satisfying $\mu\circ\pi=t^{\ell(\mu,\pi)}\mu\cdot\pi$ for each $\pi\in P_{n,m}$, $\mu\in P_{m,l}$.

\end{definition}

To see that composition is associative, it is enough to show $\nu\circ(\mu\circ\pi)=(\nu\circ\mu)\circ\pi$ for all $\pi\in P_{n,m}$, $\mu\in P_{m,l}$, $\nu\in P_{l,k}$.  To do so, consider the partition $$\includegraphics{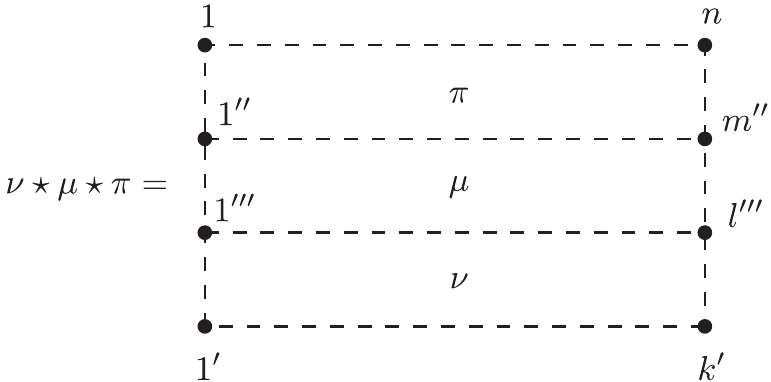}$$
Notice that $\nu\cdot(\mu\cdot\pi)$, $(\nu\cdot\mu)\cdot\pi\in P_{n,k}$ are both obtained by restricting the partition $\nu\star\mu\star\pi$ to the set $\{1,\ldots,n,1',\ldots,k'\}$.  Furthermore, $\ell(\mu,\pi)+\ell(\nu,\mu\cdot\pi)$ and $\ell(\nu,\mu)+\ell(\nu\cdot\mu,\pi)$ are both the number of connected components of $\nu\star\mu\star\pi$ whose vertices are not among $\{1,\ldots,n,1',\ldots,k'\}$.  Hence composition is associative.

One can easily check that the partition in $P_{n,n}$ whose parts are all of the form $\{j, j'\}$ is the identity morphism $\id_n:[n]\to[n]$.

\begin{example} The identity morphism $\id_7:[7]\to[7]$ is given by $$~\hspace{1.45in}\includegraphics{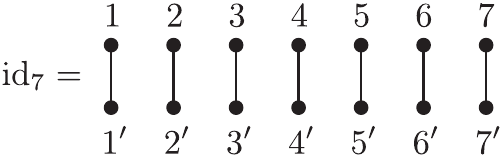}\hspace{1.4in}\Dox$$

\end{example}

Before giving $ \uRep_0(S_t; F)$ the structure of a tensor category, we pause to give the following definition which will play an important role in later sections.

\begin{definition}\label{FP} (compare with \cite{MR1103994}, \cite{MR1265453})  The \emph{partition algebra} $FP_n(t)$ is defined to be the endomorphism algebra $\End_{ \uRep_0(S_t; F)}([n])$.
\end{definition}

\begin{remark}\label{Snembedding} We identify each element of the symmetric group $S_n$ with a partition in $P_{n,n}$ as follows: $\sigma \leftrightarrow \{\{i, \sigma(i)'\}\,|\, 1\leq i\leq n\}$.  This identification extends linearly to an inclusion of algebras $FS_n\hookrightarrow FP_n(t)$ for each $t\in F$.
\end{remark}

Now to define tensor products.  While reading the following definitions, the reader may find it helpful to keep in mind the analogy studied in section \ref{motivation} between the objects $[n]$ (resp. morphisms $\pi$) in $ \uRep_0(S_t; F)$ and the objects $V_d^{\otimes n}$ (resp. morphisms $f(\pi)$) in $ \Rep(S_d; F)$.

\begin{definition}  For objects $[n], [m]$ in $ \uRep_0(S_t; F)$ set $[n]\otimes[m]:=[n+m]$.  For morphisms we let $\otimes: FP_{n_1,m_1}\times FP_{n_2, m_2}\to FP_{n_1+n_2, m_1+m_2}$ be the bilinear map such that $$\includegraphics{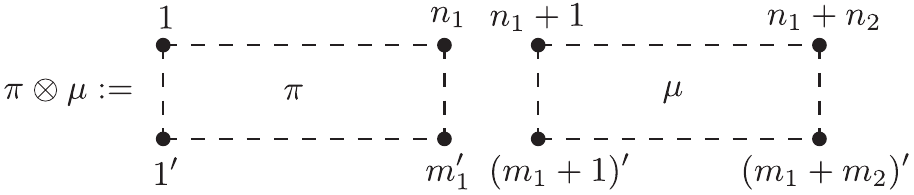}$$ for all $\pi\in P_{n_1, m_1}$, $\mu\in P_{n_2, m_2}$.
\end{definition}

\begin{example} Suppose $$\includegraphics{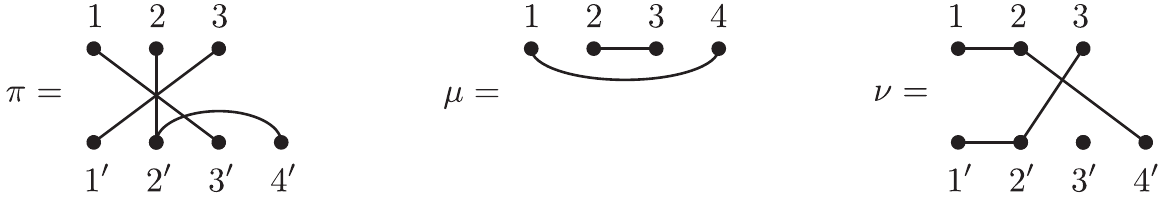}$$ Then $$\hspace{1.3in}\includegraphics{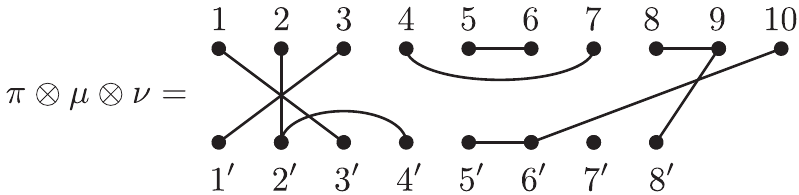}\hspace{.4in}\Dox$$
\end{example}

Next we wish to show that $ \uRep_0(S_t; F)$ is a tensor category with the following choices.
\begin{itemize}
\item (associativity)  $\alpha_{n,m,l}:([n]\otimes[m])\otimes[l] \to[n]\otimes([m]\otimes[l])$ is the identity morphism $\id_{n+m+l}$.

\item (commutativity)  $\beta_{n,m}:[n]\otimes[m]\to[m]\otimes[n]$ is the partition in $P_{n+m,n+m}$ whose parts are of the form $\{j, (m+j)'\}$ or $\{n+j, j'\}$ as illustrated below.
$$\includegraphics{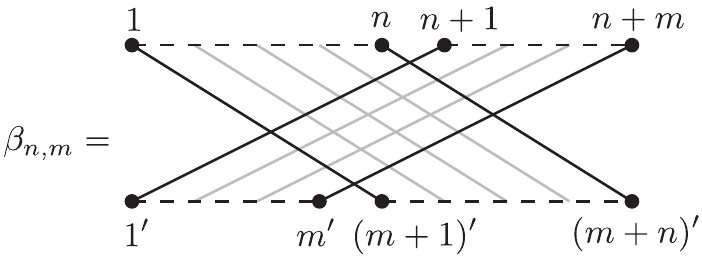}$$

\item (unit) Set ${\bf 1}:=[0]$.  Both unit morphisms $[0]\otimes[n]\to[n]$ and $[n]\otimes[0]\to[n]$ are the identity morphism $\id_n$.

\item (duals) Set $[n]^\vee:=[n]$ with the
morphism $ev_n:[n]^\vee\otimes [n]\to {\bf 1}$ (resp. $coev_n:{\bf 1}\to [n]\otimes [n]^\vee$) given by the partition in $P_{2n,0}$ (resp. $P_{0,2n}$) whose parts are of the form $\{j, n+j\}$ (resp. $\{j',(n+j)'\}$) as illustrated below.
\end{itemize}$$\includegraphics{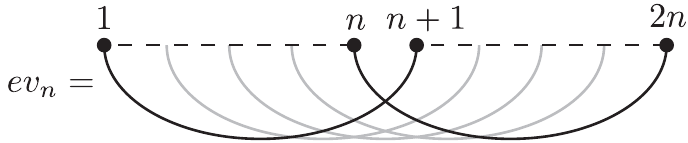}$$
$$\includegraphics{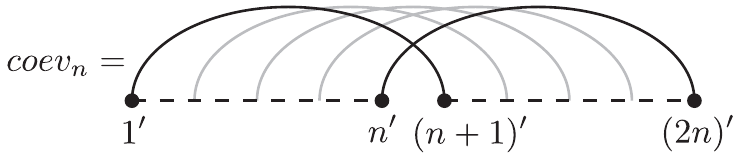}$$

The triangle and pentagon axioms are easily satisfied, as all morphisms in both diagrams are identity morphisms.  The following diagrams illustrate the hexagon axiom.
$$\includegraphics{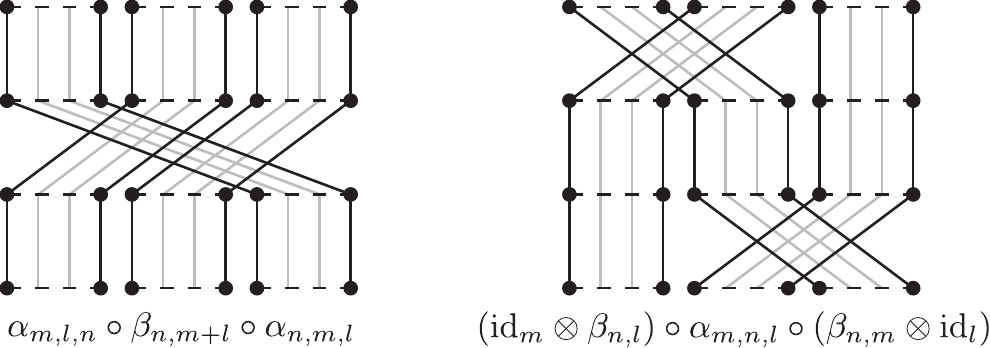}$$
The diagrams below illustrate that the choice of dual objects as well as the evaluation and coevaluation morphisms make $ \uRep_0(S_t; F)$ into a rigid category.
$$\includegraphics{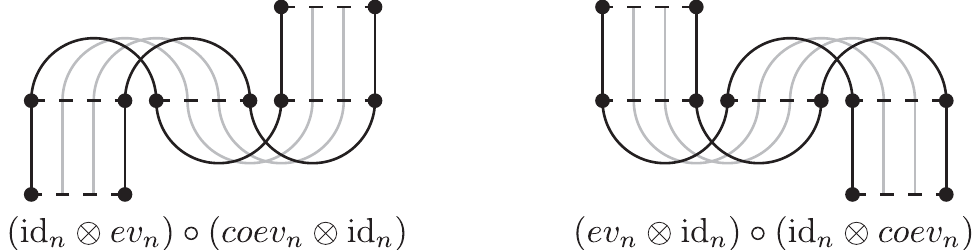}$$
Finally, $\text{End}_{ \uRep_0(S_t; F)}({\bf 1})=F$ as $P_{0,0}$ contains only the empty partition.  Thus $ \uRep_0(S_t; F)$ is a tensor category.

\begin{remark}\label{fzero} For $d\in\Z_{\geq0}$, the connection between $\uRep_0(S_d; F)$ and $\Rep_0(S_d; F)$ can be made more precise as follows.  The functor $\cat{F}_0: \uRep_0(S_d; F)\to  \Rep_0(S_d; F)$ given on objects by $[n]\mapsto V_d^{\otimes n}$, and on morphisms as the $F$-linear map $\pi\mapsto f(\pi)$ is a tensor functor.  $\cat{F}_0$ is certainly surjective on objects.  Moreover, by Theorem \ref{sw}(1), $\cat{F}_0$ is surjective on morphisms.  However, by Theorem \ref{sw}(2), $\cat{F}_0$ does not give an equivalence of categories.
\end{remark}

Now we are ready to use the process of Karoubification to define Deligne's category.  For more on this process see \cite[1.7-8]{Del07}.  Informally, $ \uRep(S_t; F)$ is the category obtained from $\uRep_0(S_t; F)$ by adding formal direct sums and images of idempotents.  Let us be more precise:

\begin{definition} Let $\uRep_1(S_t; F)$ denote the \emph{additive envelope} of $\uRep_0(S_t; F)$ so that $\uRep_1(S_t; F)$ has 

Objects:  finite tuples of objects in $\uRep_0(S_t; F)$ written as $A_1\oplus\cdots\oplus A_k$ for objects $A_1,\ldots, A_k$ in $\uRep_0(S_t; F)$.  We also include the empty tuple which is the zero object in $\uRep_1(S_t; F)$.

Morphisms:  $\Hom_{\uRep_1(S_t; F)}\left(\bigoplus_{i=1}^kA_i, \bigoplus_{i=1}^lB_i\right)$ is the set of all $l\times k$-matrices whose $(i,j)$-entry is a morphism $A_j\to B_i$ in $\uRep_0(S_t; F)$.  

Composition: given by matrix multiplication along with the induced composition from $\uRep_0(S_t; F)$.

\end{definition}

\begin{definition} Let $\uRep(S_t; F)$ denote the \emph{Karoubian envelope}\footnote{Some authors, including Deligne, refer to Karoubian envelopes as pseudo-abelian completions.} of $\uRep_1(S_t; F)$ which has 

Objects: pairs $(A, e)$ where $A$ is an object in $\uRep_1(S_t; F)$ and $e\in\End_{\uRep_1(S_t; F)}(A)$ is an idempotent.

Morphisms:  $\Hom_{\uRep(S_t; F)}\left((A,e), (B,f)\right):=f\Hom_{\uRep_1(S_t; F)}(A, B)e$.  

Composition: induced from composition in $\uRep_1(S_t; F)$.

\end{definition}

The following properties of $\uRep(S_t; F)$ follow from general theory of additive and Karoubian envelopes.

\begin{proposition}\label{proprep} (1) The obvious tensor product on objects and morphisms, as well as the obvious associativity, commutativity, unit, and dual constraints make $\uRep(S_t; F)$ a tensor category.

(2) Given a nonnegative integer $n$ and a nonzero idempotent $e\in FP_n(t)$, the object $([n], e)$ in $\uRep(S_t; F)$ is indecomposable if and only if $e$ is primitive\footnote{An idempotent is \emph{primitive} if it is nonzero and it cannot be written as a sum $e_1+e_2$ where $e_1$ and $e_2$ are nonzero idempotents with $e_1e_2=e_2e_1=0$.}.  Moreover, every indecomposable object in $\uRep(S_t; F)$ is isomorphic to one of the form $([n], e)$.  Finally, given two primitive idempotents $e, e'\in FP_n(t)$, the objects $([n], e)$ and $([n], e')$ are isomorphic if and only if $e$ and $e'$ are conjugate\footnote{Recall that two idempotents $e,e'$ in a finite dimensional algebra $A$ are
conjugate if and only if the modules $Ae$ and $Ae'$ are isomorphic.} in $FP_n(t)$.

(3) (Krull-Schmidt property) Every object in $\uRep(S_t; F)$ can be decomposed as a finite direct sum of indecomposable objects.  Moreover, if $A=L_1\oplus\cdots\oplus L_r$ and $A=L'_1\oplus\cdots\oplus L'_s$ are two decomposition of $A$ into indecomposables, then $r=s$ and there is a permutation $\sigma\in S_r$ with $L_i\cong L_{\sigma(i)}'$ for all $1\leq i\leq r$.

\end{proposition}

\begin{remark}  Informally, studying the category $ \uRep_0(S_t; F)$ is a way to simultaneously study the partition algebras $FP_n(t)$ for all $n\geq 0$ (see Definition \ref{FP}).  In this line of thinking, studying $ \uRep(S_t; F)$ is a way to simultaneously study all finitely generated projective right $FP_n(t)$-modules for all $n\geq 0$.
\end{remark}

\subsection{}\label{trace} We close this section by examining the trace of a morphism in $\uRep_0(S_t; F)$.   First, notice that $\tr:\End_{\uRep_0(S_t; F)}([n])\to F$ is an $F$-linear map.  Furthermore,  
if $\pi:[n]\to [n]$ is a partition diagram (not equal to $\id_0$) then, by the definition of trace (see \cite[3.3]{Del07}), $\tr(\pi)=t^\ell$ where $\ell$ is the number of connected components in the diagram $$\includegraphics{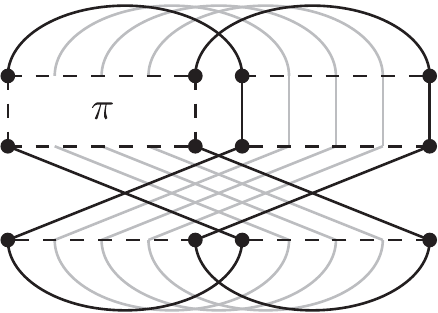}$$  Clearly $\ell$ is also the number of connected components in the following \emph{trace diagram}$$\includegraphics{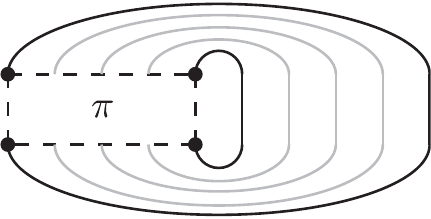}$$

\begin{example}\label{traces}(i) The only endomorphisms in $\uRep_0(S_0; F)$ with nonzero trace are nonzero scalar multiples of $\id_0$.

(ii)  In $\uRep_0(S_t; F)$, $\dim([0])=\tr(\id_0)=1$ and $\dim([n])=\tr(\id_n)=t^n$ for all positive $n$.

(iii) Consider $\pi:[7]\to [7]$ given by $$\includegraphics{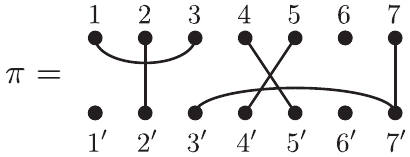}$$  $\tr(\pi)=t^4$ since there are 4 connected components in the diagram $$~\hspace{1.4in}\includegraphics{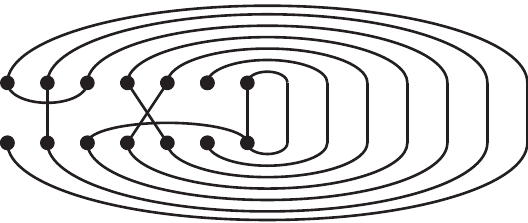}\hspace{1.3in}\Dox$$ 

\end{example}


\section{Indecomposable objects in $\uRep(S_t)$}

\subsection{Classification of indecomposable objects of $\uRep(S_t)$}\label{indecomposSt}
In this section we will classify indecomposable objects in $ \uRep(S_t; F)$ for arbitrary $t\in F$.   To do so, we first classify primitive idempotents in partition algebras.  The following lemma will be useful in that endeavor.

\begin{lemma}\label{ees} For $n>1$ let $\zeta$ denote the following 
idempotent in $FP_n(t)$:
$$\includegraphics{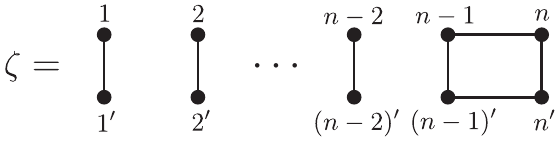}$$ Then for each  $n>1$ we have the following algebra isomorphisms:
\begin{enumerate}

\item[(1)]$\zeta FP_n(t)\zeta\cong
FP_{n-1}(t)$.

\item[(2)] $FP_n(t)/(\zeta)\cong FS_n$ where $(\zeta)$ denotes the two-sided ideal generated by $\zeta$.
\end{enumerate}
\end{lemma}

\begin{proof} To prove (1) notice we can embed $FP_{n-1}(t)$ into $FP_n(t)$ as the $F$-span of $$\{\pi\in P_{n,n}~|~n\text{ (resp. }n'\text{) is in the same part of }\pi\text{ as }n-1\text{ (resp. }(n-1)')\}.$$  This span is exactly $\zeta FP_n(t)\zeta$.  

To prove (2) recall (Remark \ref{Snembedding}) that we can embed $FS_n$ into $FP_n(t)$ by identifying  $\sigma\in S_n$ with the partition $\{\{1, \sigma(1)'\},\ldots,\{n,\sigma(n)'\}\}$.  Since $FS_n\cap (\zeta)=0$, it suffices to show a partition $\pi\in P_{n,n}$ has $\pi\in(\zeta)$ whenever $\pi\not\in S_n$.  Notice for fixed $j$ and $k$, the partition $$\includegraphics{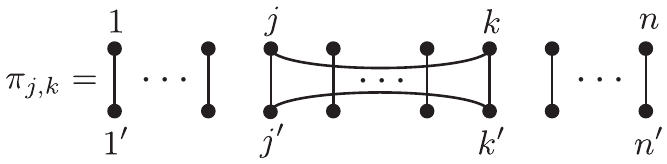}$$ is in $(\zeta)$.  Indeed, $\pi_{j,k}=\sigma \zeta\sigma$ where $\sigma\in S_n\subset P_{n,n}$ is the product of transpositions $(j,n-1)(k,n)$.  Now suppose $\mu\in P_{n,n}\setminus S_n$.  Then either $\mu$ has a part of the form $\{i\}$ for some $i\in\{1,\ldots,n\}$ or there exist $j,k\in\{1,\ldots,n\}$ which are in the same part of $\mu$.  If the latter is true, then $\mu=\mu\pi_{j,k}\in(\zeta)$.  If the former is true, then $\mu=\mu\pi_{i,j}\nu_{i,j}\in(\zeta)$ where $j\not=i$ and $$\includegraphics{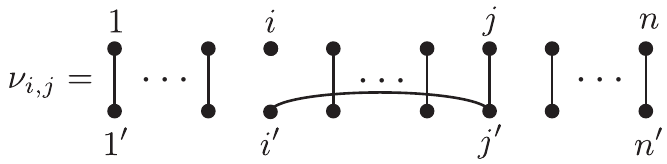}$$\end{proof}

\begin{remark} \label{Snretraction}
In fact the proof shows that the composition of the embedding $FS_n\subset FP_n(t)$
and the projection $FP_n(t)\to FP_n(t)/(\zeta)\cong FS_n$ is the identity map.
\end{remark}

Next, we state a well-known lemma (see e.g. \cite{MR1110581}) which we will use along with Lemma \ref{ees} to inductively classify primitive idempotents in partition algebras.

\begin{lemma}\label{lem} Suppose $A$ is a finite dimensional $F$-algebra and $\zeta$ is an idempotent in $A$.  As before, let $(\zeta)$ denote the two-sided ideal of $A$ generated by $\zeta$.  There is a bijective correspondence 
$$\left\{\begin{tabular}{c}
 primitive \\ idempotents\\
in $A$ up to \\ conjugation\\
\end{tabular}\right\} \stackrel{\text{bij.}}{\leftrightarrow}\left\{\begin{tabular}{c}
primitive \\ idempotents\\
in $A/(\zeta)$ up to \\ conjugation\\
\end{tabular}\right\}\sqcup\left\{\begin{tabular}{c}
primitive \\ idempotents\\
in $\zeta A\zeta$ up to \\ conjugation\\
\end{tabular}\right\}$$ satisfying the following property:

 Suppose $e$ is a primitive idempotents in $A$.  $e$ corresponds to a primitive idempotent in  $\zeta A\zeta$ if and only if $e\in(\zeta)$.  Moreover, if $e\not\in(\zeta)$ then $e$ corresponds to its image under the quotient map $A\to A/(\zeta)$.
\end{lemma}

We are now ready to classify primitive idempotents (up to conjugation) in partition algebras.  
Part (1) of the following theorem is originally due to Martin (See \cite{MR1399030}).  However, his proof does not extend to the case $t=0$.  Our proof is similar to one found in \cite{MR1779601}.

\begin{theorem} \label{idemP}
 (1) When $t\not=0$ we have the following bijection.
 $$\left\{\begin{tabular}{c}
primitive idempotents in\\
$FP_n(t)$ up to conjugation \\
\end{tabular}\right\} \stackrel{\text{bij.}}{\longleftrightarrow}\left\{\begin{tabular}{c}
 Young  diagrams $\lambda$\\
 with $|\lambda|\leq n$
\end{tabular}\right\}$$

(2)  When $n>0$ we have the following bijection.
 $$\left\{\begin{tabular}{c}
primitive idempotents in\\
$FP_n(0)$ up to conjugation \\
\end{tabular}\right\} \stackrel{\text{bij.}}{\longleftrightarrow}\left\{\begin{tabular}{c}
 Young  diagrams $\lambda$\\
 with $0<|\lambda|\leq n$
\end{tabular}\right\}$$
\end{theorem}

\begin{proof}  Part (1) is true when $n=0$ since $FP_0(t)=F$.
To show part (1) holds for $n=1$, let 
$$\includegraphics{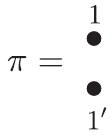}$$
and let $f$ denote the idempotent $\frac{1}{t}\pi$.
It is easy to show that $1=f+(1-f)$ is a
nontrivial decomposition of $1$ into primitive idempotents in
$FP_n(t)$ when $t\not=0$.  Thus part (1) holds when $n=1$.
Now we proceed by induction on $n$. For $n>1$ let $\zeta\in FP_n(t)$ denote the
idempotent in Lemma \ref{ees}.  
Then by Lemma \ref{ees} we have $FP_n(t)/(\zeta)\cong FS_n$ and  $\zeta FP_n(t)\zeta\cong
FP_{n-1}(t)$. Since the primitive idempotents in $FS_n$ up to conjugation
are in bijective correspondence with all Young diagrams $\lambda$
with $|\lambda|=n$, part (1) will follow by induction along
with Lemma \ref{lem}.

The proof of part (2) is similar, except $1$ is the only primitive idempotent in $FP_1(0)\cong F[\pi]/(\pi^2)$.
\end{proof}

\begin{remark}\label{zeta} Fix an integer $n>1$ and let $\zeta\in FP_n(t)$ be as in Lemma \ref{ees}.  Finally, suppose $e\in FP_n(t)$ is a primitive idempotent.  The proof of Theorem \ref{idemP} along with Lemma \ref{lem} show that $e$ corresponds to a Young diagram of size $n$ if and only if $e\not\in(\zeta)$. Moreover, if $e\not\in(\zeta)$ then the image of $e$ under the quotient map $FP_n(t)\to FP_n(t)/(\zeta)\cong FS_n$ is a primitive idempotent  corresponding to $\lambda$ in $FS_n$.
 \end{remark}

Next we classify indecomposable objects in $\uRep(S_t; F)$.  Suppose $\lambda$ is a Young diagram.  By Theorem \ref{idemP}, $\lambda$ corresponds to a primitive idempotent $e_\lambda\in FP_{|\lambda|}(t)$ (if $t=0$ set $e_\varnothing=\id_0\in FP_0(0)$).  The idempotent $e_\lambda$ is not unique, but it is unique up to conjugation; hence the object $L(\lambda)=([|\lambda|], e_\lambda)$ in $\uRep(S_t; F)$ 
is an indecomposable object which is well defined
up to isomorphism (see Proposition \ref{proprep}(2)).

\begin{lemma}\label{classifylemma}  Fix an integer $n\geq 0$.  The assignment $\lambda\mapsto L(\lambda)$ induces a bijection
$$\begin{tabular}{rcl}
$\left\{\begin{tabular}{c} 
Young diagrams $\lambda$\\
with $0\leq|\lambda|\leq n$\\
\end{tabular}\right\}$ &  $\stackrel{\text{bij.}}{\longleftrightarrow}$ & $\left\{\begin{tabular}{c} 
nonzero indecomposable objects in\\
$\uRep(S_t; F)$ 
of the form $([m], e)$\\ with $m\leq n$,
up to isomorphism\\
\end{tabular}\right\}$
\end{tabular}$$
This bijection enjoys the following properties:  

(1) If $\lambda$ is a Young diagram with $0<|\lambda|\leq n$, then there exists an idempotent $e\in FP_n(t)$ with $([n], e)\cong L(\lambda)$.

(2) If $t\not=0$, then there exists an idempotent $e\in FP_n(t)$ with $([n],  e)\cong L(\varnothing)$.  

(3) If $t=0$, then $([0], \id_0)$ is the unique object of the form $([m], e)$ which is isomorphic to $L(\varnothing)$. 

\end{lemma}

\begin{proof}  First, assume $t\not=0$.  We proceed by induction on $n$.  If $n=0$ the lemma is easy to check.  For the case $n=1$, let $\mu$ and $\mu'$ be the only elements of $P_{0,1}$ and $P_{1,0}$ respectively, and let $f\in FP_1(t)$ be the primitive idempotent from the proof of Theorem \ref{idemP}.  It is easy to check that $\{f, 1-f\}$ is a complete set of pairwise non-conjugate primitive idempotents in $FP_1(t)$.  Hence, by Proposition \ref{proprep}(2), $([1], f)$ and $([1], 1-f)$ are not isomorphic.
Moreover, $f\mu\id_0:([0], \id_0)\to([1], f)$ is an isomorphism with inverse $\frac{1}{t}\id_0\mu'f$.  Therefore  the objects  $([1], f)\cong L(\varnothing)$ and $([1], 1-f)\cong L(\Box)$  form a complete list of nonzero pairwise non-isomorphic indecomposable objects in $\uRep(S_t; F)$ of the form $([m], e)$ with $m\leq 1$.

Now suppose $n>1$.  Given a Young diagram $\nu$ with $0\leq|\nu|<n$, by induction we can find a primitive idempotent $f_\nu\in FP_{n-1}(t)$ with $([n-1], f_\nu)\cong L(\nu)$ such that $\{([n-1], f_\nu)~|~0\leq |\nu|<n\}$ is a complete set of nonzero pairwise non-isomorphic 
indecomposable objects in $\uRep(S_t; F)$ of the form $([m], e)$ with $m< n$.   For each such $\nu$, set $\hat{f}_\nu=\phi_n f_\nu\phi_n'$ where
$$\includegraphics{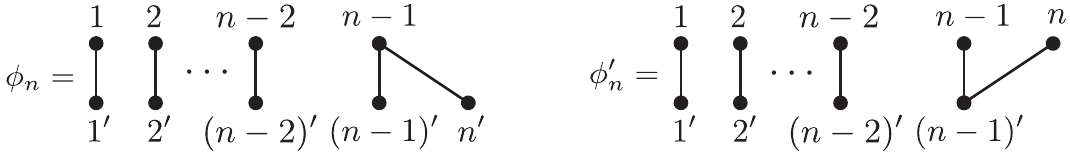}$$
Then $\hat{f}_\nu\in FP_n(t)$ is an idempotent.  Moreover, since $\phi'_n\phi_n=\id_{n-1}$, it follows that $f_\nu\phi'_n \hat{f}_\nu:([n], \hat{f}_\nu)\to([n-1], f_\nu)$ is an isomorphism with inverse $\hat{f}_\nu\phi_nf_\nu$.  In particular, by Proposition \ref{proprep}(2), $\hat{f}_\nu$ is a primitive idempotent in $FP_n(t)$.  
Moreover, $\hat{f}_\nu=\zeta\hat{f}_\nu\in(\zeta)$.  Hence, by Remark \ref{zeta}, $\hat{f}_\nu$ is not conjugate to $e_\lambda$ for any Young diagram $\lambda$ with $|\lambda|=n$.  Therefore $\{\hat{f}_\nu~|~0\leq |\nu|<n\}\cup\{e_\lambda~|~|\lambda|=n\}$ is a set of pairwise non-conjugate primitive idempotents in $FP_n(t)$.  As this set is indexed by all Young diagrams of size at most $n$, by Theorem \ref{idemP}(1), it must be a complete set of pairwise non-conjugate primitive idempotents in $FP_n(t)$.  Thus, by Proposition \ref{proprep}(2), the objects $([n], \hat{f}_\nu)\cong L(\nu)$ for $0\leq|\nu|<n$ along with $([n], e_\lambda)= L(\lambda)$ for $|\lambda|=n$ form a complete list of nonzero pairwise non-isomorphic indecomposable objects in $\uRep(S_t; F)$ of the form $([m], e)$ with $m\leq n$. 

Now assume $t=0$.  Notice every composition $([0],\id_0)\to([m], e)\to([0],\id_0)$ is equal to the zero map in $\uRep(S_0; F)$ unless $m=0$ and $e=\id_0$.  Part (3) of the lemma follows.  To prove the remainder of the lemma, we again proceed by induction on $n$.  The case $n=1$ follows from Theorem \ref{idemP}(2) along with Proposition \ref{proprep}(2).  If $n>1$, then by induction we can find idempotents $f_\nu\in FP_{n-1}(0)$ for each Young diagram $\nu$ with $0<|\nu|<n$ such that $\{([n-1], f_\nu)~|~0< |\nu|<n\}$ is a complete set of nonzero pairwise non-isomorphic indecomposable objects in $\uRep(S_0; F)$ of the form $([m], e)$ with $0<m< n$.  Define $\hat{f}_\nu$ as above.  Then the argument in the $t\not=0$ case (using part (2) of Theorem \ref{idemP} rather than part (1)) shows that $([n], \hat{f}_\nu)\cong L(\nu)$ for $0<|\nu|<n$ along with $([n], e_\lambda)= L(\lambda)$ for $|\lambda|=n$ form a complete list of nonzero pairwise non-isomorphic indecomposable objects in $\uRep(S_t; F)$ of the form $([m], e)$ with $0<m\leq n$.  
\end{proof}

The following theorem follows easily from Lemma \ref{classifylemma} (see Proposition \ref{proprep}(2)).

\begin{theorem}\label{Youngclass} The assignment $\lambda\mapsto L(\lambda)$ induces a bijection
$$\begin{tabular}{rcl}
$\left\{\begin{tabular}{c} 
Young diagrams \\
of arbitrary size\\
\end{tabular}\right\}$ & $\stackrel{\text{bij.}}{\longleftrightarrow}$ & $\left\{\begin{tabular}{c} 
nonzero indecomposable objects 
in\\
$\uRep(S_t; F)$,
up to isomorphism\\
\end{tabular}\right\}$
\end{tabular}$$
\end{theorem}

We close this section with a proposition concerning field extensions and idempotents in partition algebras which will be useful in subsequent sections.

\begin{proposition}\label{fieldex} Suppose $F\subset F'$ is a field extension and $e$ is a primitive idempotent in $FP_n(t)$. Then $e$ is also a primitive idempotent in $F'P_n(t)$.
\end{proposition}

\begin{proof} The assumption implies that $([n],e)$ is indecomposable in  $\uRep(S_t; F)$ (see Proposition \ref{proprep}(2)).
Thus by Theorem \ref{Youngclass} $([n], e)$ is isomorphic to the object $L(\lambda)=([|\lambda|],e_\lambda)$ in $\uRep(S_t; F)$ for some Young diagram $\lambda$.  Let $P_\lambda$ denote the projective $FP_{|\lambda|}(t)$-module $FP_{|\lambda|}(t)e_\lambda$.  Then the simple $FP_{|\lambda|}(t)$-module $P_\lambda/rad(P_\lambda)$ is isomorphic to the pullback of the simple $FS_{|\lambda|}$-module $L_\lambda$ through the quotient map $FP_{|\lambda|}(t)\to FP_{|\lambda|}(t)/(\zeta)\cong FS_{|\lambda|}$.  Hence  it is absolutely irreducible (since any representation of $S_{|\lambda|}$ is), so the
idempotent $e_\lambda$ is primitive in $F'P_{|\lambda|}(t)$. Hence the object $([n],e)\cong 
([|\lambda|],e_\lambda)\in \uRep(S_t; F')$ is indecomposable and we are done by Proposition \ref{proprep}(2).
\end{proof}


\subsection{Lifting objects}  Let $K$ denote the field of fractions of $F[[T-t]]$ where $T$ is an indeterminate. In this section we use the process of lifting idempotents (see appendix \ref{liftingappendix}) to lift objects in $\uRep(S_t; F)$ to objects in $\uRep(S_T; K)$.  
We then prove some useful properties of these lifted objects.  In the next section we will show that $\uRep(S_T; K)$ is semisimple (see Corollary \ref{ssT}), and hence uncomplicated from our point of view.  The process of  lifting objects developed in this section will be used later in the paper to compare the structure of $\uRep(S_t; F)$ with that of the semisimple category $\uRep(S_T; K)$.  
The following theorem will be used to define the notion of lifting objects.

\begin{theorem}\label{lifttoK} Suppose $e$ is an idempotent in $FP_n(t)$.  Then there exists an idempotent $\ep\in KP_n(T)$ of the form $\ep=\sum_{\pi\in P_{n,n}}a_\pi\pi$ with $a_\pi\in F[[T-t]]$ for all $\pi\in P_{n,n}$ such that $\ep|_{T=t}=e$ (we say $\ep$ is a lift of $e$).  Moreover, if $([n_1], e_1)$ and $([n_2], e_2)$ are isomorphic objects in $\uRep(S_t; F)$ and $\ep_1, \ep_2$ are lifts of $e_1, e_2$ respectively, then $([n_1], \ep_1)$ and $([n_2], \ep_2)$ are isomorphic in $\uRep(S_T; K)$.  
\end{theorem}

\begin{proof} The existence of a lift of $e$ is guaranteed by the first statement of Theorem \ref{appenlift}.  Now suppose $([n_1], e_1)\cong([n_2], e_2)$ in $\uRep(S_t; F)$.  To prove the remainder of the theorem we may assume $n_1\leq n_2$.  If $n_1=0$ the theorem is easy to check, so assume $n_1>0$.  Set $$\phi:=\left\{\begin{array}{ll}\id_{n_1} & \text{if }n_1=n_2 \\
\phi_{n_2}\cdots\phi_{n_1+1} & \text{if }n_1<n_2
\end{array}\right.\quad\text{and}\quad \phi':=\left\{\begin{array}{ll}\id_{n_1} & \text{if }n_1=n_2 \\
\phi'_{n_1+1}\cdots\phi'_{n_2} & \text{if }n_1<n_2
\end{array}\right.$$ where $\phi_{n}$ and $\phi'_{n}$ are as in the proof of Lemma \ref{classifylemma}.  Then $([n_1], e_1)\cong([n_2], \phi e_1\phi')$ and $\phi\ep_1\phi'$ is a lift of $\phi e_1\phi'$.  Hence it suffices to show the theorem is true when $n_1=n_2$.  
In this case the remainder of the theorem follows from the second statement of Theorem \ref{appenlift} along with Proposition \ref{proprep}(2).  
\end{proof}

Lifted idempotents satisfy the following property.

\begin{proposition}\label{liftsums} If $([n], e)\cong ([n_1], e_1)\oplus([n_2], e_2)$ in $\uRep(S_t; F)$ and $\ep, \ep_1, \ep_2$ are lifts of $e, e_1, e_2$ respectively, then $([n], \ep)\cong([n_1], \ep_1)\oplus([n_2], \ep_2)$ in $\uRep(S_T; K)$.
\end{proposition}

\begin{proof} By the Krull-Schmidt property there exist $\bar{e}_1, \bar{e}_2\in FP_n(t)$ with the property $([n_i] , e_i)\cong ([n], \bar{e}_i)$ for $i=1,2$ such that $e=\bar{e}_1+\bar{e}_2$ is an orthogonal decomposition of the idempotent $e$.  Let $\bar{\ep}_1, \bar{\ep}_2\in KP_n(T)$ be lifts of $\bar{e}_1, \bar{e}_2$ respectively.  Then  $\bar{\ep}_1+\bar{\ep}_2$ is a lift of $e$.  Hence, by Theorem \ref{appenlift} and Proposition \ref{proprep}(2) $([n], \ep)\cong([n], \bar{\ep}_1+\bar{\ep}_2)$ in $\uRep(S_T; K)$.  Since $\bar{\ep}_1$ and $\bar{\ep}_2$ are orthogonal idempotents (see Theorem \ref{appenlift}) we have $([n], \bar{\ep}_1+\bar{\ep}_2)\cong([n], \bar{\ep}_1)\oplus([n], \bar{\ep}_2)$ in $\uRep(S_T; K)$.  The result now follows from the last statement of Theorem \ref{lifttoK}.
\end{proof}

By Proposition \ref{proprep}(2) any object in $\uRep(S_t; F)$ is a finite direct sum of objects of the form $([n], e)$.  Hence, setting $\Lift_t([n], e):=([n], \ep)$, where $\ep\in KP_n(T)$ is a lift of $e\in FP_n(t)$ and extending to all objects in $\uRep(S_t; F)$ by requiring $\Lift_t(A\oplus B)=\Lift_t(A)\oplus \Lift_t(B)$ gives the operation

$$\Lift_t:\left\{\begin{tabular}{c}
objects in $\uRep(S_t; F)$\\
up to isomorphism
\end{tabular}
\right\}\longrightarrow\left\{\begin{tabular}{c}
objects in $\uRep(S_T; K)$\\
up to isomorphism
\end{tabular}
\right\}$$
Theorem \ref{lifttoK} along with Proposition \ref{liftsums} show that $\Lift_t$ is well defined.   

\begin{example}\label{lift1}  (0) $\Lift_t(L(\varnothing))=\Lift_t(([0], \id_0))=([0], \id_0)=L(\varnothing)$ for all $t\in F$.

(1) Notice $\Lift_t(([1], \id_1))=([1], \id_1)=L(\Box)\oplus L(\varnothing)$ in $\uRep(S_T; K)$ for all $t\in F$.  If $t\not=0$, then $([1], \id_1)=L(\Box)\oplus L(\varnothing)$. Hence (using part (0) and the additivity of $\Lift_t$) we have $\Lift_t(L(\Box))=L(\Box)$ whenever $t\not=0$.  On the other hand, $([1], \id_1)\cong L(\Box)$ in $\uRep(S_0; F)$.  Hence $\Lift_0(L(\Box))=L(\Box)\oplus L(\varnothing)$.
\end{example}

The following proposition lists some useful properties of $\Lift_t$.

\begin{proposition}\label{Liftprop} For this proposition suppose $A$ and $B$ are objects in $\uRep(S_t; F)$.   


(1) $\Lift_t(A\otimes B)=\Lift_t(A)\otimes\Lift_t(B)$ for all $t\in F$.

(2) $(\dim_{\uRep(S_T; K)}\Lift_t(A))|_{T=t}=\dim_{\uRep(S_t; F)}A$ for all $t\in F$.

(3) Suppose $t\in F$ and $\lambda, \lambda^{(1)}, \ldots, \lambda^{(m)}$ are Young diagrams with the property  $\Lift_t(L(\lambda))=L(\lambda^{(1)})\oplus\cdots\oplus L(\lambda^{(m)})$ in $\uRep(S_T; K)$.  Then there exists a unique $j\in\{1,\ldots, m\}$ with $\lambda^{(j)}=\lambda$.  Moreover, $|\lambda^{(i)}|<|\lambda|$ for all $i\not=j$.  In particular, $\Lift_t$ is injective. 

(4) Fix a Young diagram $\lambda$.   $\Lift_t(L(\lambda))=L(\lambda)$ for all but finitely many $t\in F$.

(5) $\dim_K\Hom_{\uRep(S_T; K)}(\Lift_t(A), \Lift_t(B))=\dim_F\Hom_{\uRep(S_t; F)}(A, B)$ for all $t\in F$.

\end{proposition}

\begin{proof}  (1) It suffices to prove the statement when $A$ and $B$ are objects of the form $A=([n], e)$, $B=([n'], e')$.  Suppose $\ep, \ep'$ are lifts of $e, e'$ respectively.  Then $\ep\otimes\ep'$ is a lift of $e\otimes e'$, hence $$\Lift_t(A\otimes B)=\Lift_t([n+n'], e\otimes e')=([n+n'], \ep\otimes\ep')=\Lift_t(A)\otimes\Lift_t(B).$$

(2) Again, we may assume $A$ is of the form $A=([n], e)$. Let  $\ep$ denote a lift of $e$.  Then $\dim_{\uRep(S_T; K)}\Lift_t(A)=\tr(\ep)$, and $\dim_{\uRep(S_t; F)}A=\tr(e)=\tr(\ep)|_{T=t}$.

(3) It is easy to show $\Lift_t(L(\varnothing))=L(\varnothing)$ for all $t\in F$, $\Lift_t(L(\Box))=L(\Box)$ for all nonzero $t\in F$, and $\Lift_0(L(\Box))=L(\varnothing)\oplus L(\Box)$.  Hence, we can assume $n:=|\lambda|>1$.   Suppose $e\in FP_n(t)$ is an idempotent and  $\ep\in KP_n(T)$ is a lift of $e$.  Let $\tilde{e}$ (resp. $\tilde{\ep}$) denote the image of $e$ (resp. $\ep$) under the quotient map $FP_n(t)\to FP_n(t)/(\zeta)\cong FS_n$ (resp.  $KP_n(T)\to KP_n(T)/(\zeta)\cong KS_n$) where $\zeta$ is as in Lemma \ref{ees}.  Then $\tilde{\ep}\in KS_n$ is a lift of $\tilde{e}\in FS_n$.  However, $\tilde{e}$ (viewed as an element of $KS_n$) is also a lift of $\tilde{e}$ (viewed as an element of $FS_n$).  Thus (by Theorem \ref{appenlift}) $\tilde{\ep}$ is conjugate to $\tilde{e}$ in $KS_n$.  Now suppose $L(\lambda)\cong([n], e)$.  Then $\tilde{e}$ is a primitive idempotent in $FS_n$ corresponding to $\lambda$ (see Remark \ref{zeta}).  Thus $\tilde{e}$ (and hence $\tilde{\ep}$) is also a primitive idempotent in $KS_n$ corresponding to $\lambda$.  Therefore, any orthogonal decomposition of $\ep$ into primitive idempotents must contain exactly one primitive idempotent corresponding to $\lambda$ and all other idempotents corresponding to Young diagrams with smaller size than $\lambda$ (see Remark \ref{zeta}).  The result follows.

(4)  Set $K':=F(T)$ (field of fractions of the polynomial ring $F[T]$) viewed as a subfield of $K$.  Now suppose $\ep\in K'P_n(T)$ is a primitive idempotent corresponding to $\lambda$.  Then by Proposition \ref{fieldex}, $\ep$ is also primitive when viewed as an element of $KP_n(T)$.  Now write $\ep=\sum_{\pi\in P_{n,n}}a_\pi \pi$ where $a_\pi\in K'$.  Let $Q\subset F$ denote the finite set consisting of all roots of all denominators of the $a_\pi$'s.  If $t\not\in Q$, then $a_\pi\in F[[T-t]]$ for all $\pi\in P_{n,n}$.  Thus, if $t\not\in Q$ then $\ep$ is a lift of $e:=\ep|_{T=t}\in FP_n(t)$.  It is easy to show that $e$ is primitive in $FP_n(t)$.  Hence $([n], e)\cong L(\mu)$ for some Young diagram $\mu$ (see Theorem \ref{Youngclass}).  Thus $\Lift_t(L(\mu))=L(\lambda)$ for all $t\not\in Q$.  By part (3) we have $\mu=\lambda$ and we are done.

(5) Set $A=([n], e_1)$, $B=([n'], e_1')$, $e_2=1-e_1\in FP_n(t)$, $e_2'=1-e_1'\in FP_{n'}(t)$, and suppose $\ep_i\in KP_n(T)$ (resp. $\ep'_i\in KP_{n'}(T)$) is a lift of $e_i$ (resp. $e'_i$) for $i\in\{1,2\}$.  Then \begin{equation}\label{ineq} \dim_K\ep_i'KP_{n,n'}\ep_j\geq \dim_Fe_i' FP_{n,n'}e_j\end{equation} for all $i,j\in\{1,2\}$ since evaluating $T=t$ certainly cannot increase dimension.  On the other hand, $$ KP_{n,n'}=\bigoplus_{1\leq i, j\leq 2}\ep'_iKP_{n,n'}\ep_j\quad\text{and}\quad FP_{n,n'}=\bigoplus_{1\leq i, j\leq 2}e'_iFP_{n,n'}e_j.$$  Comparing dimensions we get $$|P_{n,n'}|=\sum_{1\leq i,j \leq 2}\dim_K\ep'_iKP_{n,n'}\ep_j\geq\sum_{1\leq i,j \leq 2}\dim_Fe'_iFP_{n,n'}e_j=|P_{n,n'}|.$$  Hence equality must hold in (\ref{ineq}) for all $i, j\in\{1,2\}$.  In particular, equality holds in (\ref{ineq}) when $i=j=1$ which gives the desired statement.  
\end{proof}


\subsection{On $ \uRep(S_t; F)$ for generic $t$}\label{genericSt}  In this section we will show that $ \uRep(S_t; F)$ is semisimple for ``generic" values of $t$.  Deligne showed that $ \uRep(S_t; F)$ is not semisimple if and only if $t$ is a nonnegative integer (see \cite{Del07}).  That result will follow from our description of the blocks in $ \uRep(S_t; F)$ (see Corollary \ref{ss}).  For now we confine ourselves to prove a weaker result (see Theorem \ref{generic}) which allows us to conclude that $\uRep(S_T; K)$ is semisimple (see Corollary \ref{ssT}).

The following well-known lemma will be useful in showing that $\uRep(S_t; F)$ is generically semisimple.

\begin{lemma}\label{trform}  Suppose $A$ is a finite dimensional $F$-algebra.  For $a\in A$, let $\phi_a$ denote the $F$-linear map $A\to A$ given by $x\mapsto ax$.  Define the \emph{trace form} on $A$ by $(a, b):=\tr(\phi_a\phi_b)$.  Then $A$ is a semisimple algebra if and only if the trace form on $A$ is non-degenerate.  
\end{lemma}

\begin{proof} Let $S:=\{a\in A~|~(a, b)=0\text{ for all }b\in A\}$.  We will show that $S$ is equal to $J(A)$ (the Jacobson radical of $A$).  Suppose $a\in J(A)$.  Since $J(A)$ is a nilpotent ideal, $ab$ is nilpotent for every $b\in A$.  Hence $(a, b)=0$ for all $b\in A$, so $J(A)\subset S$.  To show $S\subset J(A)$ it suffices to show that every element of $S$ is nilpotent.  
If $a\in S$ then $\tr(\phi_a^n)=0$ for all integers $n>0$.   If we let $x_1,\ldots, x_r\in F$ denote the eigenvalues of $\phi_a$, then we have $\sum_{i=1}^rx_i^n=0$ for all $n>0$.  Hence, any non-constant symmetric polynomial in $x_1,\ldots, x_r$ is zero (see for example \cite{MR1354144}).  Thus the characteristic polynomial of $\phi_a$ must be $\chi(X)=X^r$.  
Therefore $\phi_a$, and thus $a$, is nilpotent.  
\end{proof}

Before showing $\uRep(S_t; F)$ is generically semisimple we give two examples illustrating the usefulness of Lemma \ref{trform}.  The reader may find these examples helpful when reading the proof of Theorem \ref{generic}. 

\begin{example} (1) Consider the partition algebra $FP_1(t)$. Let $\pi$ be as in the proof of Theorem \ref{idemP} and fix the ordered basis $P_{1,1}=\{\id_1, \pi\}$ for $FP_1(t)$.  Under this ordered basis 
 the matrix with entries $(x, y)$ for $x, y\in P_{1,1}$ is $$\left(\begin{array}{cc} 2 & t\\ t & t^2\\\end{array}\right)$$
Since the determinant of the matrix above is $t^2$, we conclude (by Lemma \ref{trform}) that $FP_1(t)$ is semisimple if and only if $t\not=0$.

(2)  Consider the partition algebra $FP_2(t)$.  Using the ordered basis $$\includegraphics{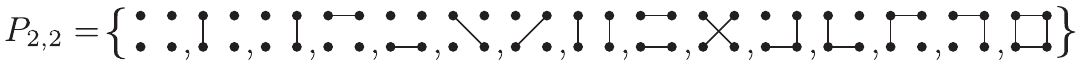}$$ the matrix with entries $(\pi, \mu)$ for $\pi, \mu\in P_{2,2}$ is 
$${\scriptsize \left(\begin{array}{ccccccccccccccc}
2t^4 & 2t^3 & 2t^3 & 2t^3 & 2t^3 & 2t^3 & 2t^3 & 2t^2 & 2t^2 & 2t^2 & 2t^2 & 2t^2 & 2t^2 & 2t^2 & 2t\\
2t^3& 5t^2& 2t^2& 2t^2& 2t^2& 2t^2& 2t^2& 5t& 2t& 2t& 2t& 5t& 5t& 2t& 5\\
2t^3& 2t^2& 5t^2& 2t^2& 2t^2& 2t^2& 2t^2& 5t& 2t& 2t& 5t& 2t& 2t& 5t& 5\\
2t^3& 2t^2& 2t^2& 2t^2& 2t^3& 2t^2& 2t^2& 2t& 2t^2& 2t& 2t^2& 2t^2& 2t& 2t& 2t\\
2t^3& 2t^2& 2t^2& 2t^3& 2t^2& 2t^2& 2t^2& 2t& 2t^2& 2t& 2t& 2t& 2t^2& 2t^2& 2t\\
2t^3& 2t^2& 2t^2& 2t^2& 2t^2& 2t^2& 5t^2& 2t& 2t& 5t& 5t& 2t& 5t& 2t& 5\\
2t^3& 2t^2& 2t^2& 2t^2& 2t^2& 5t^2& 2t^2& 2t& 2t& 5t& 2t& 5t& 2t& 5t& 5\\
2t^2& 5t& 5t& 2t& 2t& 2t& 2t& 15& 2t& 7& 5& 5& 5& 5& 5\\
2t^2& 2t& 2t& 2t^2& 2t^2& 2t& 2t& 2t& 2t^2& 2t& 2t& 2t& 2t& 2t& 2t\\
2t^2& 2t& 2t& 2t& 2t& 5t& 5t& 7& 2t& 15& 5& 5& 5& 5& 5\\
2t^2& 2t& 5t& 2t^2& 2t& 5t& 2t& 5& 2t& 5& 5& 5& 2t& 5t& 5\\
2t^2& 5t& 2t& 2t^2& 2t& 2t& 5t& 5& 2t& 5& 5& 5& 5t& 2t& 5\\
2t^2& 5t& 2t& 2t& 2t^2& 5t& 2t& 5& 2t& 5& 2t& 5t& 5& 5& 5\\
2t^2& 2t& 5t& 2t& 2t^2& 2t& 5t& 5& 2t& 5& 5t& 2t& 5& 5& 5\\
2t& 5& 5& 2t& 2t& 5& 5& 5& 2t& 5& 5& 5& 5& 5& 5\\
\end{array}\right)}$$ The determinant of the matrix above is $1259712t^{14}(t-1)^4(t-2)^6$.  Hence, by Lemma \ref{trform}, $FP_2(t)$ is semisimple if and only if $t\not=0, 1, 2$.\hfill$\Dox$ 
\end{example}

Now to prove that $\uRep(S_t; F)$ is generically semisimple.

\begin{theorem}\label{generic} $ \uRep(S_t; F)$ is semisimple for all but countably many values of $t$.  Moreover, if $ \uRep(S_t; F)$ is not semisimple, then $t$ is an algebraic integer.
\end{theorem}

\begin{proof}  We may assume $t\not=0$.  Suppose $L_1$ and $L_2$ are two indecomposable objects in $\uRep(S_t; F)$.  By Proposition \ref{proprep}(2) along with Lemma \ref{classifylemma}(1)(2) we can find a nonnegative integer $n$ and idempotents $e_1, e_2\in FP_n(t)$ so that $L_1\cong([n], e_1)$ and $L_2\cong([n], e_2)$.  Whence $\Hom_{ \uRep(S_t; F)}(L_1, L_2)=e_2FP_n(t)e_1$.  Thus, in order to show $\Hom_{ \uRep(S_t; F)}(L_1, L_2)$ is either zero or a finite dimensional division algebra over $F$, it suffices to show $FP_n(t)$ is a semisimple algebra.  Therefore, for a fixed $t\in F$, $ \uRep(S_t; F)$ is semisimple whenever $FP_n(t)$ are semisimple for all $n\geq 0$.  

Let $M_n(t)$ denote the matrix whose rows and columns are labelled by the elements of $P_{n,n}$ (in some fixed order) with the $x,y$-entry equal to $(x, y)$ (the trace form on $FP_n(t)$, see Lemma \ref{trform}).  Then the entries of $M_n(t)$ are in $\Z[t]$.  Hence $\det M_n(t)\in\Z[t]$.  It follows from Lemma \ref{trform} that $FP_n(t)$ is semisimple if and only if $\det M_n(t)\not=0$.  However, from Theorem \ref{sw}, we know that $FP_n(d)$ is semisimple for integers $d\geq 2n$.  Hence, $\det M_n(t)$ is a polynomial in $t$ which is not identically zero.
Thus, for each $n\geq0$ there are only finitely many values of $t$ for which $\det M_n(t)=0$.  Therefore there are only countably many values of $t$ for which $FP_n(t)$ is not semisimple for all $n\geq 0$.   
\end{proof}

\begin{remark} If $t\in F$ is not an algebraic integer, then by Theorem \ref{generic} we know $ \uRep(S_t; F)$ is semisimple.  However, given an arbitrary $t\in F$, 
neither Theorem \ref{generic} nor its proof allow us to determine if $\uRep(S_t; F)$ is semisimple.   As mentioned at the beginning of this section, we will eventually show that $\uRep(S_t; F)$ is semisimple if and only if $t$ is not a nonnegative integer.
\end{remark}

We close this section with one final observation.

\begin{corollary}\label{ssT}  $\uRep(S_T; K)$ is semisimple.
\end{corollary}

\begin{proof} This follows from Theorem \ref{generic} as $T$ is not an algebraic integer.\end{proof}


\subsection{The interpolation functor $\uRep(S_d; F)\to \Rep(S_d; F)$}\label{interpolationSt} Throughout this section we assume $d$ is a nonnegative integer.
In this section we will describe following \cite[\S 6]{Del07} how $\uRep(S_d; F)$ ``interpolates" the category $\Rep(S_d; F)$.  More precisely, we will show $\Rep(S_d; F)$ is equivalent to the quotient of $\uRep(S_t; F)$ by the so-called ``negligible morphisms."  To start, let us define the \emph{interpolation functor}.

\begin{definition}\label{Ffunc} $\cat{F}: \uRep(S_d; F)\to  \Rep(S_d; F)$ is the functor defined on indecomposable objects by $\cat{F}([n], e)=f(e)(V_d^{\otimes n})$, and on morphisms $\alpha:([n], e)\to([n'], e')$ by $\cat{F}(\alpha)=f(\alpha)$.  Here $f$ and $V_d$ are as in section \ref{motivation}.
\end{definition}

Notice that $\cat{F}$ is clearly a tensor functor.  From the discussion in section \ref{motivation} we have the following.

\begin{proposition}\label{fully}  $\cat{F}$ is surjective on objects and morphisms.
\end{proposition}

\begin{proof}
This follows from Proposition \ref{full} and Theorem \ref{sw}(1).
\end{proof}

However, by Theorem \ref{sw}(2) we know that $\cat{F}$ does not induce an equivalence of categories.  To illustrate the amount by which $\cat{F}$ fails to induce an equivalence of categories we need the following definition.

\begin{definition} A morphism $f:X\to Y$ in a tensor category is called \emph{negligible} if $\tr(fg)=0$ for all $g:Y\to X$.  Set $\cat{N}(X, Y):=\{f:X\to Y~|~f\text{ is negligible}\}.$
\end{definition}

\begin{example}\label{neg} (1)  The only morphisms in $\uRep_0(S_0; F)$ which are not negligible are nonzero scalar multiples of $\id_0$ (see Example \ref{traces}(1)). 

(2) Let $\pi:[1]\to[1]$ be as in the proof of Theorem \ref{idemP}.  Then $x_\pi:[1]\to[1]$, defined by equation (\ref{xpi}), is given by $x_\pi=\pi-\id_1$.  Since $\tr(\pi)=t=\tr(\id_1)$, we have $\tr(x_\pi)=0$.  Moreover, $x_\pi\pi=(t-1)\pi$ so that $\tr(x_\pi\pi)=t(t-1)$.  We conclude that $x_\pi$ is negligible if and only if $t=0, 1$.

(3) Consider the morphism $\pi\in\Hom_{\uRep_0(S_t; F)}([1], [2])$ given by $$\includegraphics{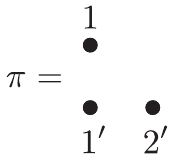}$$  From equation (\ref{xpi}) we get $x_\pi=\pi-\mu_1-\mu_2-\mu_3+2\mu_4$ where $$\includegraphics{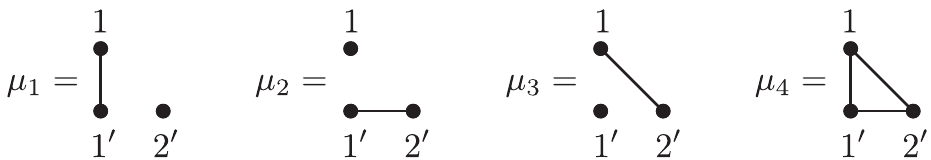}$$  If we set $$\includegraphics{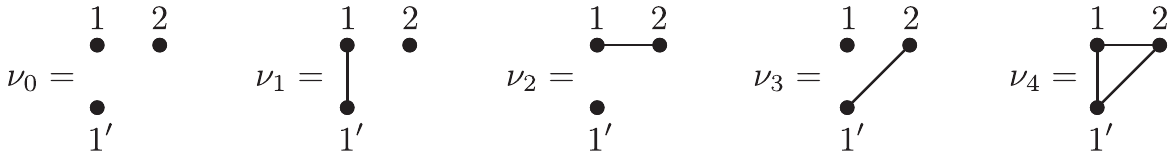}$$
then one can compute $\tr(x_\pi\nu_0)=t(t-1)(t-2)$, and $\tr(x_\pi\nu_i)=0$ for $i=1, 2, 3, 4$.  Thus $x_\pi$ is negligible if and only if $t=0, 1, 2$.\hfill$\Dox$
\end{example}

\begin{remark}  Each of the examples \ref{neg} follow from the following fact:  In $\uRep_0(S_t; F)$ $$\cat{N}([n], [m])=\left\{\begin{array}{ll}
\text{Span}_F\{x_\pi~|~\pi\in P_{n,m}\text{ has more than $t$ parts}\}, & \text{if }t\in\Z_{\geq 0}\\
0, & \text{otherwise.}
\end{array}\right.$$
For $t\in\Z_{\geq 0}$ this fact follows from Theorem \ref{sw}(2), Remark \ref{fzero}, along with the following Proposition \ref{N}(2).  For $t\not\in\Z_{\geq0}$ we will eventually show that the larger category $\uRep(S_t; F)$ is semisimple (see Corollary \ref{ss}).  It is well-known that there are no nonzero negligible morphisms in a semisimple category.
\end{remark}

\begin{proposition}\label{N} The following statements hold in any tensor category.

(1) $\cat{N}$ is a tensor ideal.

(2) The image under a full tensor functor of a morphism $f$ is negligible if and only if $f$ is negligible.
\end{proposition}

\begin{proof} Suppose $f:A\to B$, $g:B\to A$, and $h:B\otimes C\to A\otimes C$ are morphisms in a tensor category.  Observe that $\tr(f\circ g)=\tr(g\circ f)$.  Moreover, one can show $\tr((f\otimes \id_C)\circ h)=\tr(f\circ(\id_A\otimes ev_C)\circ (h\otimes\id_C)\circ(\id_B\otimes coev_C))$.  Statement (1) follows.  Statement (2) follows from the fact that a tensor functor preserves the trace of a morphism.
\end{proof}

Since there are no nonzero negligible morphisms in $\Rep(S_d; F)$, by Proposition \ref{N}(2) we conclude the functor $\cat{F}:\uRep(S_d; F)\to \Rep(S_d; F)$ sends all negligible morphisms to zero.  Thus $\cat{F}$ induces a functor $\cat{\overline{F}}:\uRep(S_d; F)/\cat{N}\to \Rep(S_d; F)$.

\begin{theorem} $\cat{\overline{F}}$ induces an equivalence of categories $\uRep(S_d; F)/\cat{N}\cong \Rep(S_d; F)$.
\end{theorem}

\begin{proof} This follows from Proposition \ref{fully} and Proposition \ref{N}(2).
\end{proof}

We finish this section with the following proposition concerning the functor $\cat{F}$.  For proof of the proposition we refer the reader to \cite[Proposition 6.4]{Del07}\footnote{Deligne assumes $d\geq 2|\lambda|$ in \cite[Proposition 6.4]{Del07}.  If that assumption is removed, Deligne's proof will show Proposition \ref{Findy}.}.

\begin{proposition}\label{Findy}  Suppose $d$ is a nonnegative integer and $\lambda=(\lambda_1, \lambda_2, \ldots)$ is a Young diagram. 
If $d-|\lambda|\geq \lambda_1$, then $\cat{F}(L(\lambda))=L_{\lambda(d)}$.  If $d-|\lambda|<\lambda_1$, then $\cat{F}(L(\lambda))=0$.

\end{proposition}


\subsection{Dimensions}\label{dim} In this section we study a hook length formula which gives the dimension of indecomposable objects in $\uRep(S_T; K)$.    Let us start by defining the hook length formula.

\begin{definition}\label{hookpoly}  The \emph{hook length} of a fixed box in a Young diagram $\lambda$ is the number of boxes in $\lambda$ which are either directly below or directly to the right of the fixed box, counting the fixed box itself once.
 Given a Young diagram $\lambda$ let $P_\lambda$ denote the unique polynomial such that $$P_\lambda(d)=\frac{d!}{\prod(\text{hook lengths of }\lambda(d))}$$ for every integer $d\geq2|\lambda|$.
\end{definition}

\begin{example}  Let $\lambda=\includegraphics[bb= 0 4 16 1]{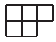}$ and suppose $d\geq 10$ is an integer.  In the following picture each box of $\lambda(d)$ is labeled by its hook length.  
$$\includegraphics{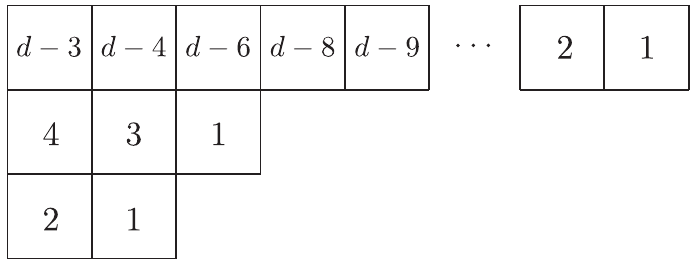}$$
Thus $P_\lambda(d)=\frac{d!}{24(d-3)(d-4)(d-6)(d-8)!}=\frac{1}{24}d(d-1)(d-2)(d-5)(d-7)$.\hfill$\Dox$
\end{example}

First, we show how $P_\lambda$ is related to the indecomposable object in $\uRep(S_T; K)$ corresponding to $\lambda$.

\begin{proposition}\label{gendimen} $\dim_{\uRep(S_T; K)}L(\lambda)=P_\lambda(T)$ for any Young diagram $\lambda$.
\end{proposition}

\begin{proof}  Fix a Young diagram $\lambda$.  Let $K'=F(T)$ and suppose $\ep\in K'P_{|\lambda|}(T)$ is a primitive idempotent with $L(\lambda)=([|\lambda|], \ep)$ in $\uRep(S_T; K')$.  Applying Proposition \ref{fieldex} to the field extension $K'\subset K$ shows $L(\lambda)=([|\lambda|], \ep)$ in $\uRep(S_T; K)$ too.  Thus  $\dim_{\uRep(S_T; K)}L(\lambda)=\tr(\ep)$ is a rational function in $T$.   Now, by Proposition \ref{Liftprop}(4) and Proposition \ref{Findy} we can find an integer $N$ so that $\Lift_d(L(\lambda))=L(\lambda)$ and $\cat{F}:\uRep(S_d; F)\to \Rep(S_d; F)$ has $\cat{F}(L(\lambda))=L_{\lambda(d)}$ for all integers $d>N$.  Hence, if $d>N$ is an integer, we have $$\begin{array}{rll}
P_\lambda(d)\hspace{-.1in} & = \dim_{\Rep(S_d; F)}L_{\lambda(d)} & \text{(see e.g. \cite[4.12]{MR1153249})}\\
 & =\dim_{\uRep(S_d; F)}L(\lambda) & \text{(since tensor functors preserve dimension)}\\
 & =(\dim_{\uRep(S_T; K)}L(\lambda))|_{T=d}  & \text{(by Proposition \ref{Liftprop}(2))}
\end{array}$$
Hence, $\dim_{\uRep(S_T; K)}L(\lambda)$ (a rational function in $T$) agrees with $P_\lambda(T)$ (a polynomial in $T$) for infinitely many values of $T$ .  Hence they must always agree.
\end{proof}

Next, we wish to determine the roots of $P_\lambda$.  The following combinatorics will be useful toward that endeavor.  

\begin{definition} Given a Young diagram $\lambda$ and an integer $d\geq2|\lambda|$, create the \emph{$(\lambda, d)$ grid marking} as follows: 
Start with a grid of $(|\lambda|+1)\times (d-|\lambda|)$ black boxes.  Place the Young diagram $\lambda(d)$ (with white boxes) atop the grid so that the upper left corner of $\lambda(d)$ is atop the upper left corner of the grid.  Now place the numbers $0,\ldots, d-1$ into the boxes of the grid using the following rules:
 \begin{itemize}
\item Begin by placing the number 0 in the lower left box of the grid.
\item If the number $i$ is in a black box, place $i+1$ into the box directly above $i$.
\item If the number $i$ is in a white box, place $i+1$ into the box directly to the right of $i$.
\end{itemize} Label the rows of the grid $0,\ldots, |\lambda|$ (from top to bottom) and the columns of the grid $1,\ldots, d-|\lambda|$ (from left to right).  
\end{definition}

\begin{example} Set $\lambda=(4, 3, 1, 1, 0,\ldots)$ and $d=25$.  Below is the $\lambda$-marking of the $10\times 16$ grid. $$~\hspace{.75in}\includegraphics{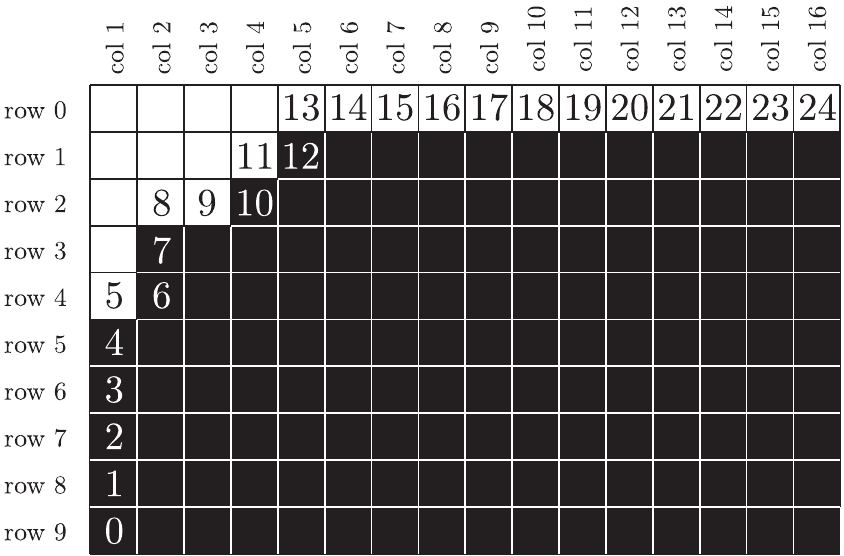}\hspace{.7in}\Dox$$
\end{example}

The following proposition records the properties of the $(\lambda, d)$ grid marking which will be useful for determining the roots of $P_\lambda$.

\begin{proposition}\label{grid} The $(\lambda, d)$ grid marking has the following properties:
\begin{enumerate}
\item[(1)] If $\includegraphics[bb= 0 4 12 1]{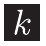}$ appears in the $i$th row, then $k=|\lambda|+\lambda_i-i$.
\item[(2)] If $\includegraphics[bb= 0 4 12 1]{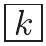}$ appears in the $i$th column, then $d-k$ is the hook length of the row $0$, column $i$ box in $\lambda(d)$.
\end{enumerate}
\end{proposition}

\begin{proof} (1)  If $\includegraphics[bb= 0 4 12 1]{diss52.pdf}$ appears in the $i$th row, then $\includegraphics[bb= 0 4 12 1]{diss52.pdf}$ must be in column $\mu_i+1$.  Hence $k$ is the number of up/right moves it takes to get from the lower left corner to the upper right corner in a $(|\lambda|-i)\times(\lambda_i+1)$ grid.  The result follows.

(2)  Let $c_i$ denote the number of boxes in the $i$th column of $\lambda$.  Notice there are $c_i$ boxes below, and $d-|\lambda|-i$ boxes to the right of the row $0$, column $i$ box in $\lambda(d)$.  Hence the hook length of that box is $d-|\lambda|-i+c_i+1$.  On the other hand, if $\includegraphics[bb= 0 4 12 1]{diss53.pdf}$ appears in the $i$th column, then $\includegraphics[bb= 0 4 12 1]{diss53.pdf}$ must be in row $c_i$.  Hence $k$ is the number of up/right moves it takes to get from the lower left corner to the upper right corner in a $(|\lambda|-c_i)\times i$ grid.  Thus $k=|\lambda|-c_i+i-1$.
\end{proof}

Using Proposition \ref{grid}, we can determine all the roots of the polynomial $P_\lambda$.  

\begin{proposition}\label{roots} $P_\lambda$ is a degree $|\lambda|$ polynomial with $|\lambda|$ distinct, integer roots given by $|\lambda|+\lambda_i-i$ for each $i=1,\ldots,|\lambda|$.
\end{proposition}

\begin{proof} Suppose $d$ is an integer with $d>2|\lambda|$.  It follows from Definition \ref{hookpoly} that the roots of $P_\lambda$ are exactly the integers $0\leq k<d$ such that $d-k$ is not a hook length of a box in the top row of $\lambda(d)$.  By Proposition \ref{grid}(2), those are exactly the values of $k$ for which $\includegraphics[bb= 0 4 12 1]{diss52.pdf}$ appears in the $(\lambda, d)$ grid marking.  The result now follows from Proposition \ref{grid}(1).
\end{proof}


\section{Endomorphisms of the identity functor} \label{endId}

In this section we study endomorphisms of the identity functor on $\uRep_0(S_t; F)$ constructed using certain central elements in group algebras of symmetric groups.  These endomorphisms of the identity functor will play a key role in describing the blocks in $\uRep(S_t; F)$.  This role is analogous to the role the Casimir element plays in Lie theory.


\subsection{Interpolating sums of $r$-cycles} In this section we define morphisms in $\uRep_0(S_t; F)$ which ``interpolate" the action of the sum of all $r$-cycles on representations of symmetric groups.  To begin, let $r$ and $d$ be positive integers with $r\leq d$. 
\begin{definition}\label{sums} Let $\Omega_{r, d}\in FS_d$ denote the sum of all $r$-cycles in $S_d$.\end{definition}  
Since $\Omega_{r, d}$ is in the center of $FS_d$, the action of $\Omega_{r, d}$ on $V_d^{\otimes n}$ gives an element of $\End_{S_d}(V_d^{\otimes n})$ for each integer $n\geq 0$.  
This, along with Theorem \ref{sw}, shows the following definition is valid.

\begin{definition}\label{Cint} For nonnegative integers $r, n,$ and $d$ with $r\leq d$ and $2n\leq d$, let $C_n^{r}( d)$ denote the unique element of $FP_n(d)$ with $f(C_n^r(d))\in\End_{S_d}(V_d^{\otimes n})$ given by the action of $\Omega_{r, d}$.
\end{definition}

The first goal of this section is to define elements of $FP_n(t)$ for arbitrary $t\in F$ which agree with the definition of $C_n^r(t)$ when $t$ is a sufficiently large integer.  We are able to do this because, as we will show, $C_n^r(d)$ depends polynomially on $d$.  The fact that $C_n^r(d)$ depends polynomially on $d$ boils down to the following combinatorial proposition.

\begin{proposition}\label{rcycs} Suppose $n$ is a nonnegative integer and $\pi\in P_{n, n}$.  Fix the following notation.
\begin{itemize}
\item Let $a$ denote the number of parts of $\pi$.

\item Let $b$ denote the number of parts $\pi_k$ of $\pi$ such that $j, j'\in \pi_k$ for some $1\leq j\leq n$.

\item Let $c$ denote the number of connected components in the trace diagram of $\pi$ (see section \ref{trace}).

\item Suppose $r$ and $d$ are positive integers with $d\geq r$, and ${\bd{i}}, {\bd{i}}'\in[n, d]$ are such that the $({\bd{i}}, {\bd{i}}')$-coloring of $\pi$ is perfect.  Let $S(\pi, r, d)$ denote the number of $r$-cycles, $\sigma\in S_d$, such that $\sigma(i_j)=i'_j$ for all $1\leq j\leq n$.

\end{itemize}
Fix a positive integer $r$.  

(1) If $S(\pi, r, d)\not=0$ for some integer $d\geq r$, then $S(\pi, r, d')\not=0$ for all $d'\geq d$.  

(2) If $S(\pi, r, d)$ is nonzero, then \begin{equation}\label{rcycles}S(\pi, r, d)=\frac{(r-a+c-1)!}{(r-a+b)!}\prod_{k=1}^{r+b-a}(d-r-b+k).\footnote{If $r=a-b$ then the empty product $\prod\limits_{k=1}^{r+b-a}(d-r-b+k)$ is equal to 1, and (\ref{rcycles}) gives $S(\pi, r, d)=(c-b-1)!$.  If $r<a-b$ or $r\leq a-c$ then $S(\pi, r, d)$ is zero for all $d$, so (\ref{rcycles}) does not apply.}\end{equation}
\end{proposition}

\begin{proof}  Part (1) is clear.  To prove part (2), since $S(\pi, r, d)$ does not depend on the choice of perfect coloring of $\pi$, we may assume $\{i_j, i_j'~|~j=1,\ldots, n\}=\{1,\ldots, a\}$.  For $x, y\in\{1,\ldots, a\}$ write $x\to y$ if $x=i_j$ and $y=i_j'$ for some $1\leq j\leq n$.  Generate the weakest equivalence relation on $\{1,\ldots, a\}$ such that $x$ and $y$ are equivalent whenever $x\to y$.  Notice that the equivalence classes correspond precisely to the connected components in the trace diagram of $\pi$.  Hence there are $c$ equivalence classes.  Now assume $S(\pi, r, d)\not=0$.  Then the following two implications must hold: if $x\to y$ and $x\to y'$ then $y=y'$; if $x\to y$ and $x'\to y$ then $x=x'$.
In particular, there are exactly $b$ equivalence classes with only one element; hence the $c-b$ equivalence classes with more than one element must account for $a-b$ elements of $\{1,\ldots, a\}$.  

Now, an $r$-cycle in $\sigma\in S_d$ can be thought of as a cyclic arrangement of $r$ distinct elements of $\{1,\ldots, d\}$.  If $\sigma(i_j)=i_j'$ for all $1\leq j\leq n$, then among the elements in $\{1,\ldots, a\}$ 
the cyclic arrangement corresponding to $\sigma$ contains precisely the $a-b$ elements in equivalence classes with more than one element.   Moreover, if $x\to y$ for distinct $x, y\in\{1,\ldots, a\}$, then $x$ and $y$ must be adjacent in the cyclic arrangement corresponding to $\sigma$.  Hence $\sigma$ is determined by a cyclic arrangement of the following $r-a+c$ items:  $c-b$ equivalence classes with more than one element;  $r-a+b$ elements of $\{a+1,\ldots, d\}$.  Finally, if such a $\sigma$ exists, then any choice of $r-a+b$ elements of $\{a+1,\ldots, d\}$ arranged cyclically with the equivalence classes containing more than one element determines an $r$-cycle in $S_d$ which maps $i_j\mapsto i_j'$.  Since there are $(r-a+c-1)!$ cyclic arrangements of $r-a+c$ items, and ${d-a\choose r-a+b}$ different $(r-a+b)$-element subsets of $\{a+1,\ldots, d\}$, it follows that $S(\pi, r, d)=(r-a+c-1)!{d-a\choose r-a+b}$, which is equivalent to the desired formula.
\end{proof}

\begin{example}  Let $n=15$ and $$\includegraphics{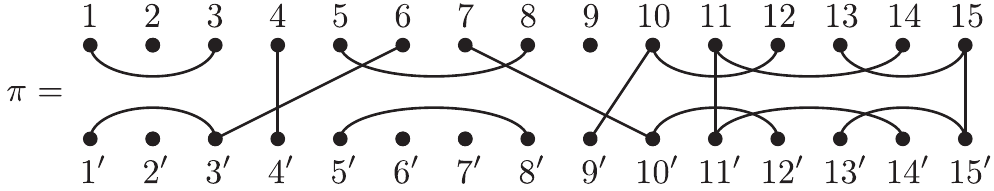}$$ Then $a=14$, $b=3$, and $c=7$.  For $d\geq 14$, let ${\bd{i}}, {\bd{i}}'\in[n, d]$ be the functions which give the perfect $({\bd{i}}, {\bd{i}}')$-coloring of $\pi$ below.
$$\includegraphics{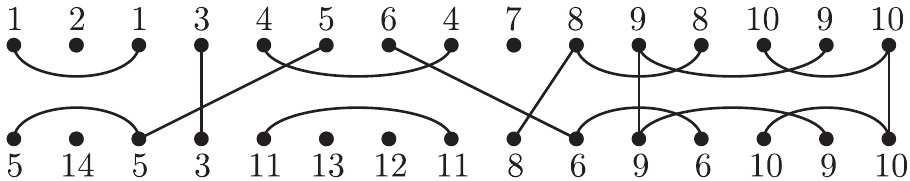}$$ For an $r$-cycle $\sigma\in S_d$ to satisfy $\sigma(i_j)=i'_j$ for $j=1,\ldots, n$, $\sigma$ must fix the $3=b$ numbers $3, 9, 10$, and map
\begin{equation}\label{chains}
1\to 5\to 13,\qquad 2\to 14,\qquad 4\to 11,\qquad 7\to 8\to 6\to 12.\end{equation}
Clearly, such an $r$-cycle exists if and only if $r\geq11=a-b$ and $d\geq r+3=r+b$.  In that case, the number of such $r$-cycles can be counted as follows.  
There are $(r-8)!=(r-a+c-1)!$ ways to arrange the $4=c-b$ ``chains" listed in (\ref{chains}) along with the $r-11=r-a+b$ other entries within the cycle.  Moreover, there are $d-14=d-a$ choices for the remaining $r-11=r-a+b$ entries within the $r$-cycle after the ``chains" in (\ref{chains}) have been taken into account.  Hence the number of desired $r$-cycles is given by $(r-8)!{d-14\choose r-11}=(r-a+c-1)!{d-a\choose r-a+b}$ which agrees with (\ref{rcycles}).\hfill$\Dox$
\end{example}

With Proposition \ref{rcycs} in mind, we are now ready for the following definition.

\begin{definition}\label{omega}  For $t\in F$ and integers $r>0$, and $n\geq0$, define $\omega_n^r(t)\in FP_n(t)$ as follows.  Using the basis $\{x_\pi\}$ for $FP_n(t)$ defined in (\ref{xpi}), set $$\omega_n^r(t)=\sum_{\pi\in P_{n,n}}q_{\pi, r, t}x_\pi$$ where, using the notation set up in Proposition \ref{rcycs}, $$q_{\pi, r, t}=\left\{\begin{array}{ll}
0,& \text{if }S(\pi, r, d)=0\text{ for all integers }d>1,\\
\frac{(r-a+c-1)!}{(r-a+b)!}\prod\limits_{k=1}^{r+b-a}(t-r-b+k), & \text{otherwise.}\end{array}\right.$$
\end{definition}

Although the definition of $\omega_n^r(t)$ may seem a bit complicated, for the rest of this paper we will only be concerned with the following (less complicated) properties of $\omega_n^r(t)$.

\begin{proposition}\label{centralomega} (1) Fix integers $r>0$ and $n\geq 0$.  Whenever $d$ is a sufficiently large\footnote{The statement is certainly true for $d\geq 2n+r$, although this bound is not sharp.} integer, $\omega_n^r(d)=C_n^r(d)$.  In other words, when $d$ is a sufficiently large integer, the map $f(\omega_n^r(d)):V_d^{\otimes n}\to V_d^{\otimes n}$ is given by the action of $\Omega_{r, d}\in S_d$.

(2) Fix $t\in F$ and an integer $r>0$.  The morphisms $\omega_n^r(t): [n]\to[n]$ for each nonnegative integer $n$ form an endomorphism of the identity functor on $\uRep_0(S_t; F)$.  In particular, $\omega_n^r(t)$ is in the center of $FP_n(t)$ for every $t\in F$ and integer $n\geq 0$.

\end{proposition}

\begin{proof}  For ${\bd{i}}, {\bd{i}}'\in [n, d]$, let $\pi({{\bd{i}},{\bd{i}}'})\in P_{n,n}$ denote the unique partition  which has a perfect $({\bd{i}},{\bd{i}}')$-coloring.  Then the action of $\Omega_{r, d}$ on $V_d^{\otimes n}$ maps the basis vector $v_{\bd{i}}\mapsto \sum_{{\bd{i}}'\in [n, d]}S(\pi({\bd{i}},{\bd{i}}'), r, d)v_{{\bd{i}}'}$.  On the other hand, (\ref{fx}) shows that $f(\omega_n^r(t))$ maps $v_{\bd{i}}\mapsto \sum_{{\bd{i}}'\in [n, d]}q_{\pi({\bd{i}},{\bd{i}}'), r, d}v_{{\bd{i}}'}$.  By Proposition \ref{rcycs}, $S(\pi({\bd{i}},{\bd{i}}'), r, d)=q_{\pi({\bd{i}},{\bd{i}}'), r, d}$ for sufficiently large $d$.  This proves part (1).

To prove part (2), choose $\mu\in P_{n,m}$.  For an integer $d>r$, we know that $f(\mu):V_d^{\otimes n}\to V_d^{\otimes m}$ commutes with the action of $\Omega_{r, d}\in S_d$.  Hence, by part (1),  $f(\omega_m^r(d)\mu)=f(\mu\omega_n^r(d))$ when $d$ is a sufficiently large integer.  Thus, by part (2) of Theorem \ref{sw}, $\omega_m^r(d)\mu=\mu\omega_n^r(d)$ when $d$ is a sufficiently large integer.
If we set $\omega_m^r(t)\mu=:\sum_{\pi\in P_{n, m}}a_\pi(t)\pi$ and $\mu\omega_n^r(t)=:\sum_{\pi\in P_{n, m}}a'_\pi(t)\pi$ for each $t\in F$, then we have shown the polynomials $a_\pi(t)$ and $a'_\pi(t)$ are equal when $t$ is a sufficiently large integer.  Hence they are always equal.
\end{proof}


\subsection{Frobenius' formula} This section will be devoted to studying how $\omega^r_n(t)$ interacts with indecomposable objects in $\uRep(S_t; F)$.  We start with the following proposition.

\begin{proposition}\label{xi}  Fix $t\in F$ along with integers $r>0$ and $n\geq0$.  If $e$ is a primitive idempotent in $FP_n(t)$, then there exist $\xi\in F$ and a positive integer $m$ such that $(\omega_n^r(t)-\xi)^me=0$.
\end{proposition}

\begin{proof}  Let $\bar{F}$ denote the algebraic closure of $F$ and write $\omega:=\omega_n^r(t)$.  Let $a(x)$ (resp. $a'(x)$) denote the monic polynomial of minimal degree in $F[x]$ (resp. $\bar{F}[x]$) with $a(\omega)e$ (resp. $a'(\omega)e$) equal to zero.\footnote{The polynomial $a(x)$ (resp. $a'(x)$) exists since $FP_n(t)$ (resp. $\bar{F}P_n(t)$) is finite dimensional over $F$ (resp. $\bar{F}$).}  First we will show that $a(x)$ is a power of an irreducible polynomial in $F[x]$.  To do so, suppose $b(x)$ and $c(x)$ are relatively prime monic polynomials in $F[x]$ with $a(x)=b(x)c(x)$.  Then there exist polynomials $g(x), h(x)\in F[x]$ with $\deg g(x)<\deg c(x)$, $\deg h(x)<\deg b(x)$, and $g(x)b(x)+h(x)c(x)=1$.  Hence $g(\omega)b(\omega)e+h(\omega)c(\omega)e=e$ is a decomposition of $e$ into orthogonal idempotents (here we are using the fact that $\omega$ is in the center of $FP_n(t)$, see Proposition \ref{centralomega}(2)).  Since $e$ is primitive, this implies $g(\omega)b(\omega)e=0$ or $h(\omega)c(\omega)e=0$.  The minimality of $a(x)$ implies that either $g(x)=0$ or $h(x)=0$, which implies $c(x)=1$ or $b(x)=1$.  Thus $a(x)$ is a power of an irreducible polynomial in $F[x]$.  Since $e$ is primitive in $\bar{F}P_n(t)$ (see Proposition \ref{fieldex}), the same line of reasoning shows $a'(x)$ is a power of an irreducible polynomial in $\bar{F}[x]$.  Hence $a'(x)=(x-\xi)^m$ for some positive integer $m$ and $\xi\in\bar{F}$. Since $a(x)$ is a power of an irreducible polynomial in $F[x]$ and $(x-\xi)^m$ divides $a(x)$ in $\bar{F}[x]$, we conclude $\xi\in F$.
\end{proof}

Next, we use a classical result of Frobenius to produce a formula for the scalar $\xi$ in Proposition \ref{xi}.  The study of this formula will be the key to describing the blocks of $\uRep(S_t, F)$ in section \ref{secblocks}.  First we state Frobenius' result on the symmetric group:

\begin{theorem}[Frobenius' formula\footnote{We take formula (\ref{frob}) to be the definition of $\xi_{r, k}^\lambda$ even if $\lambda=(\lambda_0, \lambda_1,\ldots)$  is not a Young diagram.}]\label{fro} Fix integers $d\geq r\geq 1$.  Given a Young diagram $\lambda=(\lambda_0, \lambda_1,\ldots)$ of size $d$, set $\mu_i=\lambda_i-i$ for each $i\geq 0$.  Then $\Omega_{r, d}$, (Definition \ref{sums}), acts on the simple $S_d$-module corresponding to $\lambda$ by the scalar 
\begin{equation}\label{frob}\xi^\lambda_{r, k}:=\frac{1}{r}\sum_{i=0}^k~(\mu_i+k-1)(\mu_i+k-2)\cdots(\mu_i+k-r)\prod_{0\leq j\leq k\atop j\not=i}\frac{\mu_i-\mu_j-r}{\mu_i-\mu_j}
\end{equation} where $k$ is any positive integer such that $\lambda_{k+1}=0$.
\end{theorem}

The result of Theorem \ref{fro}  first appeared in \cite{Frob}.  A modern proof of Theorem \ref{fro} is outlined in \cite[exercise 4.17]{MR1153249}, (see also \cite[Example 7 in section I.7]{MR1354144}).

We close this section by showing the scalar $\xi$ in Proposition \ref{xi} is given by Frobenius' formula.

\begin{proposition}\label{minpoly} Fix $t\in F$, a positive integer $r$, and a Young diagram $\lambda$.  Suppose that $L(\lambda)=([n], e)$ in $\uRep(S_t, F)$.   If $k$ is a positive integer such that $\lambda_{k+1}=0$, then $(\omega_n^r(t)-\xi^{\lambda(t)}_{r, k})^me=0$ for some positive integer $m$.
\end{proposition}

\begin{proof} By Theorem \ref{Youngclass} we may assume $n=|\lambda|$.  Let $\xi$ and $m$ be as in Proposition \ref{xi}, so that $(\omega_n^r(t)-\xi)^me=0$.  Applying the quotient map $\psi:FP_n(t)\twoheadrightarrow FS_n$ (Lemma \ref{ees}(2)) to this equation yields $(\psi(\omega_n^r(t))-\xi)^mc_\lambda=0$ in $FS_n$ where $c_\lambda$ is a primitive idempotent in $FS_n$ corresponding to $\lambda$.  Since $\psi(\omega_n^r(t))$ is central in $FS_n$, this implies $\psi(\omega_n^r(t))c_\lambda=\xi c_\lambda$.  Hence, by Definition \ref{omega}, $\xi$ depends polynomially on $t$.  

Now, assume $d$ is a positive integer such that $d\geq\lambda_1+|\lambda|$.  By Proposition \ref{Findy}, applying the functor $\cat{F}$ to the equation $(\omega_n^r(d)-\xi)^me=0$ yields $(\Omega_{r, d}-\xi)^mc_{\lambda(d)}=0$.  Thus, by Theorem \ref{fro}, $\xi=\xi^{\lambda(t)}_{r, k}$ whenever $t=d$ is a sufficiently large integer.  Since $\xi$ depends polynomially on $t$, $\xi^{\lambda(t)}_{r, k}$ is a rational function in $t$, and $\xi=\xi^{\lambda(t)}_{r, k}$ for infinitely many values of $t$, we conclude that $\xi=\xi^{\lambda(t)}_{r, k}$ for all $t\in F$.
\end{proof}


\section{Blocks}\label{blocks}   

\subsection{Blocks in additive categories}  \label{blocksdef}
Let $\cat{A}$ denote an arbitrary $F$-linear Karoubian category.  Consider the weakest equivalence relation on the set of isomorphism classes of indecomposable objects in $\cat{A}$ where two indecomposable objects are equivalent whenever there exists a nonzero morphism between them.  We call the equivalence classes in this relation \emph{blocks}\setcounter{footnote}{0}\footnote{Notice that while this definition makes sense for a general additive category it is reasonable only for categories satisfying suitable finiteness assumptions.
At the very least one needs to require that any object decomposes into a finite direct sum of indecomposable ones.}.  We will also use the term block to refer to a full subcategory of $\cat{A}$ whose objects are direct sums of indecomposable objects in a single block.  We will say a block is \emph{trivial} if it contains only one indecomposable object (up to isomorphism) and the endomorphism ring of that indecomposable object is $F$.  

Consider the following equivalence class on the set of Young diagrams of arbitrary size.

\begin{definition}\label{sim}  For $t\in F$ and a Young diagram $\lambda=(\lambda_1, \lambda_2, \ldots)$, set  $$\mu_\lambda(t):=(t-|\lambda|, \lambda_1-1, \lambda_2-2,\ldots).$$  For Young diagrams $\lambda$ and $\lambda'$ write $\mu_\lambda(t)=(\mu_0, \mu_1,\ldots)$ and $\mu_{\lambda'}(t)=(\mu'_0, \mu'_1,\ldots)$.  We write $\lambda\stackrel{t}{\sim}\lambda'$ whenever there exists a bijection $\tau:\Z_{\geq 0}\to\Z_{\geq 0}$ with $\mu_i=\mu'_{\tau(i)}$ for all $i\geq 0$.
\end{definition}

\begin{example} Let $$\includegraphics{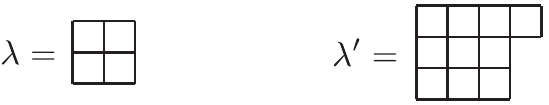}$$ Then $\mu_\lambda(7)=(3, 1, 0, -3, -4, -5,\ldots)$ and $\mu_{\lambda'}(7)=(-3, 3, 1, 0, -4, -5,\ldots)$.  Hence $\lambda\stackrel{7}{\sim}\lambda'$.\hfill$\Dox$
\end{example}

Clearly, for each $t\in F$,  $\stackrel{t}{\sim}$ defines an equivalence relation on the set of all Young diagrams, and hence on the indecomposable objects in $\uRep(S_t; F)$ (see Theorem \ref{Youngclass}).
The main goal of section \ref{blocks} is devoted to the proof of the following theorem.

\begin{theorem}\label{blocksinSt} $L(\lambda)$ and $L(\lambda')$ are in the same block of $\uRep(S_t; F)$ if and only if $\lambda\stackrel{t}{\sim}\lambda'$.
\end{theorem}


\subsection{}  In this section we use Frobenius' formula to show that $\stackrel{t}{\sim}$-equivalence classes are unions  of blocks in $\uRep(S_t; F)$.

\begin{lemma}\label{equalxi} Suppose $\lambda$ and $\lambda'$ are Young diagrams and $k>0$ is an integer with $\lambda_{k+1}=\lambda'_{k+1}=0$.
 
 (1)  If $L(\lambda)$ and $L(\lambda')$ are in the same block in $\uRep(S_t; F)$, then $\xi^{\lambda(t)}_{r, k}=\xi^{\lambda'(t)}_{r, k}$ for every $r>0$.

(2) If $\xi^{\lambda(t)}_{r, k}=\xi^{\lambda'(t)}_{r, k}$ for every $r>0$, then $\lambda\stackrel{t}{\sim}\lambda'$.
\end{lemma}

\begin{proof} To prove part (1), let us first fix some notation.  Let $n, n'$ be nonnegative integers and $e\in FP_n(t), e'\in FP_{n'}(t)$ be idempotents with $L(\lambda)\cong([n], e)$ and $L(\lambda')\cong([n'], e')$.  Fix $r$ and write $\xi:=\xi^{\lambda(t)}_{r, k}$, $\xi':=\xi^{\lambda'(t)}_{r, k}$, $\omega:=\omega_n^r(t)$, $\omega':=\omega_{n'}^r(t)$.  Finally, let $m$ be a positive integer with $(\omega-\xi)^me=0$ and $(\omega'-\xi')^me'=0$ (such an $m$ exists by Proposition \ref{minpoly}).  Now, suppose $\xi\not=\xi'$.  Then there are polynomials $p(x), q(x)\in F[x]$ with $p(x)(x-\xi)^m+q(x)(x-\xi')^m=1$.  Hence, given any morphism $\phi:([n'], e')\to([n], e)$ in $\uRep(S_t; F)$, we have $$\phi=p(\omega)(\omega-\xi)^m\phi+q(\omega)(\omega-\xi')^m\phi=p(\omega)(\omega-\xi)^me\phi+q(\omega)(\omega-\xi')^m\phi e'.$$  By  Proposition \ref{centralomega}(2), the right side of the equation above is equal to $$p(\omega)(\omega-\xi)^me\phi+\phi q(\omega')(\omega'-\xi')^m e'=0.$$  Thus, if there exists a nonzero morphism $([n'], e')\to([n], e)$ in $\uRep(S_t; F)$, then $\xi=\xi'$.

To prove part (2), first notice that $\xi^{\lambda(t)}_{r, k}$ is symmetric in $\mu_0,\ldots, \mu_k$.  Thus $\xi^{\lambda(t)}_{r, k}\prod_{0\leq i<j\leq k}(\mu_i-\mu_j)$ is an antisymmetric polynomial in $\mu_0,\ldots, \mu_k$.  However, every antisymmetric polynomial in $\mu_0,\ldots, \mu_k$ is divisible by $\prod_{0\leq i<j\leq k}(\mu_i-\mu_j)$.  Thus $\xi^{\lambda(t)}_{r, k}$ is a symmetric polynomial in $\mu_0,\ldots, \mu_k$.  
Moreover, examining equation (\ref{frob}) shows that as a polynomial in $\mu_0,\ldots, \mu_k$, $$\xi^{\lambda(t)}_{r, k}=\frac{1}{r}\sum_{i=0}^k\mu_i^r+\left(\text{terms of total degree less than }r\right).$$  The formula above shows that the ring generated by $\{\xi_{r,k}^{\lambda(t)}\}_{r>0}$ contains the power sum $\sum_{i=0}^k\mu_i^r$ for each $r>0$ and hence is equal to the ring of all symmetric polynomials in $\mu_0,\ldots, \mu_k$.  Thus, if $\xi^{\lambda(t)}_{r, k}=\xi^{\lambda'(t)}_{r, k}$ for all $r>0$ then $\mu_0,\ldots, \mu_k$ is a permutation of $\mu'_0,\ldots, \mu'_k$.
\end{proof}


\subsection{On the equivalence relation $\stackrel{t}{\sim}$}  In this section we prove some elementary properties of the equivalence relation $\stackrel{t}{\sim}$ and give examples.

We say a $\stackrel{t}{\sim}$-equivalence class is \emph{trivial} if it contains only one Young diagram.  Soon we will see that the $\stackrel{t}{\sim}$-equivalence classes are all trivial unless $t$ is a nonnegative integer (see Corollary \ref{gensim}(1)).  First, we prove the following easy proposition.

\begin{proposition}\label{gs}  Suppose $\lambda$ is a Young diagram and write $\lambda_\mu(t)=(\mu_0, \mu_1,\ldots)$.  Suppose further that $\tau:\Z_{\geq 0}\to\Z_{\geq 0}$ is a bijection and set $\mu'=(\mu'_0, \mu'_1,\ldots)$ where $\mu'_{i}=\mu_{\tau^{-1}(i)}$.  There exists a Young diagram $\lambda'$ such that $\mu'=\mu_{\lambda'}(t)$ if and only if $\mu_i'\in\Z$ with $\mu'_i>\mu'_{i+1}$ for all $i>0$.  
\end{proposition}

\begin{proof} Suppose $\lambda'$ satisfies $\mu'=\mu_{\lambda'}(t)$. Then $\mu'_i=\lambda_i'-i>\lambda_{i+1}'-i-1=\mu'_{i+1}$ for all $i>0$.  

On the other hand, suppose $\mu'_i>\mu'_{i+1}$ for all $i>0$.  Set $\lambda_i'=\mu_i+i$ for all $i\geq0$ and $\lambda'=(\lambda_1', \lambda_2',\ldots)$.  Since $\mu_i>\mu_{i+1}$ and $\mu'_i>\mu'_{i+1}$ for all $i>0$, $\tau(i)$ must equal $i$ for all $i>\text{max}\{\tau(0), \tau^{-1}(0)\}$.  Thus $\lambda'_i=\lambda_i$ for all $i>\text{max}\{\tau(0), \tau^{-1}(0)\}$.  This shows $\lambda_i'=0$ for all but finitely many values of $i$.  Moreover, $\mu_i>\mu_{i+1}$ for $i>0$ implies $\lambda_i'\geq\lambda_{i+1}'$ for all $i>0$.  Thus $\lambda'$ is indeed a Young diagram.   Finally, choose $k>\text{max}\{\tau(0), \tau^{-1}(0)\}$ with $\lambda_k=0$.  Then $\lambda_k'=0$ as well.  Furthermore, $t=\sum_{i=0}^k\lambda_i=\sum_{i=0}^k\mu_i+\frac{k(k+1)}{2}=\sum_{i=0}^k\mu'_i+\frac{k(k+1)}{2}=\sum_{i=0}^k\lambda_i'$, which implies $\lambda_0'=t-|\lambda'|$.
\end{proof}

\begin{corollary}\label{gensim} (1) The $\stackrel{t}{\sim}$-equivalence classes are all trivial unless $t\in\Z_{\geq0}$.

(2) Suppose $d\in\Z_{\geq0}$ and $\lambda$ is a Young diagram. The $\stackrel{d}{\sim}$-equivalence class containing $\lambda$ is nontrivial if and only if the coordinates of $\mu_\lambda(d)$ are pairwise distinct.
\end{corollary}

\begin{proof} Suppose $\lambda$ is a Young diagram.  It follows from Proposition \ref{gs} that the $\stackrel{t}{\sim}$-equivalence class containing $\lambda$ is nontrivial if and only if the coordinates of $\mu_\lambda(t)$ are all integers which are pairwise distinct.  This proves part (2).  Also, this implies the $\stackrel{t}{\sim}$-equivalence classes are all trivial unless $t\in\Z$.  Finally, if $t$ is a negative integer, then $\lambda_{|\lambda|-t}=0$ which implies $\mu_{|\lambda|-t}=\mu_0$.  This proves part (1). 
\end{proof}

\begin{example} (1) Let $\lambda=(2, 1, 0,\ldots)$.  Then $\mu_\lambda(t)=(t-3, 1, -1, -3, -4,\ldots)$.  Thus, by Corollary \ref{gensim}, $\lambda$ is in a nontrivial $\stackrel{t}{\sim}$-equivalence class if and only if $t\in\Z_{\geq 0}$ and $t\not=0, 2, 4$.  

(2)  If we let $\varnothing$ denote the Young diagram $(0,\ldots)$, then $\mu_\varnothing(t)=(t, -1, -2, -3,\ldots)$.  Thus, by Corollary \ref{gensim}, $\lambda$ is in a nontrivial $\stackrel{t}{\sim}$-equivalence class if and only if $t\in\Z_{\geq 0}$.~\hfill$\Dox$
\end{example}

The following proposition will allow us to give a complete description of nontrivial $\stackrel{t}{\sim}$-equivalence classes.

\begin{proposition}\label{mindiagprop} Suppose $d$ is a nonnegative integer and $\lambda$ is a Young diagram in a nontrivial $\stackrel{d}{\sim}$-equivalence class.  Label the coordinates of $\mu_\lambda(d)$ by $\mu_0, \mu_1,\ldots$ so that $\mu_{i}<\mu_{i-1}$ for all $i>0$ (such a labeling is possible by Corollary \ref{gensim}(2)).  For each integer $i\geq0$, set $\lambda^{(i)}=(\lambda^{(i)}_1, \lambda^{(i)}_2,\ldots)$ where \begin{equation}\label{lambdai}\lambda_j^{(i)}:=\left\{\begin{array}{ll}
\mu_{j-1}+j &\text{if $j\leq i$},\\
\mu_j+j & \text{if $j>i$}.
\end{array}\right.\end{equation} for each $j>0$.   Then $\lambda^{(0)}\subset\lambda^{(1)}\subset\lambda^{(2)}\subset\cdots$ is a complete list of the Young diagrams $\stackrel{d}{\sim}$-equivalent to $\lambda$.
\end{proposition}

\begin{proof}   Given $i\geq0$,  $\lambda^{(i)}$ is a Young diagram with $\mu_{\lambda^{(i)}}(d)=(\mu_0^{(i)}, \mu_1^{(i)},\ldots)$ where $$\mu_j^{(i)}=\left\{\begin{array}{ll}
\mu_{i} &\text{if $j=0$},\\
\mu_{j-1} & \text{if $0<j\leq i$},\\
\mu_j &  \text{if $j> i$}.
\end{array}\right.$$  Hence $\lambda^{(i)}\stackrel{d}{\sim}\lambda$ for all $i\geq0$.  Moreover, it follows from Proposition \ref{gs} that $\lambda^{(0)}, \lambda^{(1)}, \lambda^{(2)},\ldots$ is a complete list of Young diagrams which are $\stackrel{d}{\sim}$-equivalent to $\lambda$.  
Furthermore,  
$\lambda^{(i-1)}_j=\lambda^{(i)}_j$ for all $j\not=i$ and $\lambda^{(i-1)}_{i}=\mu_i+i<\mu_{i-1}+i=\lambda^{(i)}_i$ for each $i>0$.  Hence $\lambda^{(i-1)}\subset\lambda^{(i)}$ for all $i>0$.  
\end{proof}

\begin{corollary}\label{mindiag} A Young diagram $\lambda$ is the minimal element in a nontrivial $\stackrel{d}{\sim}$-equivalence class if and only if $\lambda(d)$ is a Young diagram.  In particular, the nontrivial $\stackrel{d}{\sim}$-equivalence classes are parametrized by Young diagrams of size $d$.  Moreover, if $\{\lambda^{(0)}\subset\lambda^{(1)}\subset\cdots\}$ is a nontrivial $\stackrel{d}{\sim}$-equivalence class, then for each $i\geq 0$ the Young diagram $\lambda^{(i)}$ is obtained from the Young diagram $\lambda^{(0)}(d)$ by removing row $i$ and adding one box to row $j$ for each $j=0,\ldots, i-1$.
\end{corollary}

\begin{proof} Suppose $\lambda$ is a Young diagram with the property that $\lambda(d)$ is also a Young diagram.  Then $d-|\lambda|>\lambda_1-1>\lambda_2-2>\cdots$.  Thus, by Corollary \ref{gensim}(2),  $\lambda$ is in a nontrivial $\stackrel{d}{\sim}$-equivalence class.  Moreover, $\lambda=\lambda^{(0)}$ (using the notation in Proposition \ref{mindiagprop}) is the minimal element in its $\stackrel{d}{\sim}$-equivalence class.  Conversely, suppose $\{\lambda^{(0)}\subset\lambda^{(1)}\subset\cdots\}$ is a nontrivial $\stackrel{d}{\sim}$-equivalence class and let $\mu_0,\mu_1,\ldots$ be as in Proposition \ref{mindiagprop}.  From (\ref{lambdai}) we have  $d-|\lambda^{(0)}|=\mu_0\geq\mu_1+1=\lambda_1^{(0)}$.  Hence $\lambda^{(0)}(d)$ is a Young diagram.    Finally, using (\ref{lambdai}) we see $\lambda^{(i)}_1=\mu_0+1=d-|\lambda^{(0)}|+1$, 
$\lambda^{(i)}_j=\mu_{j-1}+j=\lambda^{(0)}_{j-1}+1$ for $0<j\leq i$, and $\lambda^{(i)}_{j}=\lambda^{(0)}_j$ for $j>i$.  The result follows.
\end{proof}

\begin{example}\label{3class}  (1) The only nontrivial $\stackrel{0}{\sim}$-equivalence class is $\{\varnothing, (1), (1^2), (1^3),\ldots\}$.

(2) Below are the three nontrivial $\stackrel{3}{\sim}$-equivalence classes of Young diagrams. $$\includegraphics{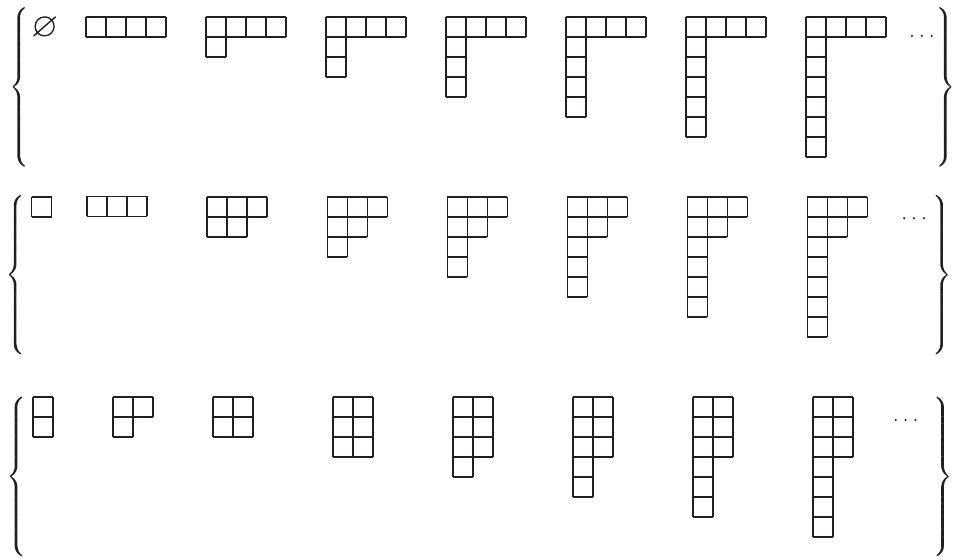}$$
~\hfill$\Dox$
\end{example}

Next we show how the polynomials $P_\lambda$ defined in section \ref{dim} can be used to determine when a Young diagram is in a trivial $\stackrel{d}{\sim}$-equivalence class.

\begin{proposition}\label{zero} Suppose $\lambda$ is a Young diagram and $d$ is a nonnegative integer.  $\lambda$ is in a trivial $\stackrel{d}{\sim}$-equivalence class if and only if $P_\lambda(d)=0$.
\end{proposition}

\begin{proof} By Corollary \ref{gensim}, $\lambda$ is in a trivial $\stackrel{d}{\sim}$-equivalence class if and
only if the coordinates of $\mu_\lambda(d)$ are not distinct, which occurs if and only if $d-|\lambda|=\lambda_i-i$ for some $i>0$.  However, $d-|\lambda|>\lambda_i-i$ when $i>|\lambda|$, so $\lambda$ is in a trivial $\stackrel{d}{\sim}$-equivalence class if and only if $d=|\lambda|+\lambda_i-i$ for some $0<i\leq|\lambda|$.  The result now follows from Proposition \ref{roots}.
\end{proof}

We conclude this section by defining a total order on the nontrivial $\stackrel{d}{\sim}$-equivalence classes.  This ordering will be useful in section \ref{secblocks}.

\begin{definition} If $B$ and $B'$ are nontrivial $\stackrel{d}{\sim}$-equivalence classes with minimal diagrams $\lambda$ and $\lambda'$ respectively, we write $B\prec B'$ if $\lambda(d)\prec\lambda'(d)$ (see section \ref{notation}).
\end{definition}

\begin{example} (1) The nontrivial $\stackrel{3}{\sim}$-equivalence classes in Example \ref{3class}(2) are listed in decreasing order.

(2)  Below are the seven nontrivial $\stackrel{5}{\sim}$-equivalence classes with $B_0\prec\cdots\prec B_6$. $$\includegraphics{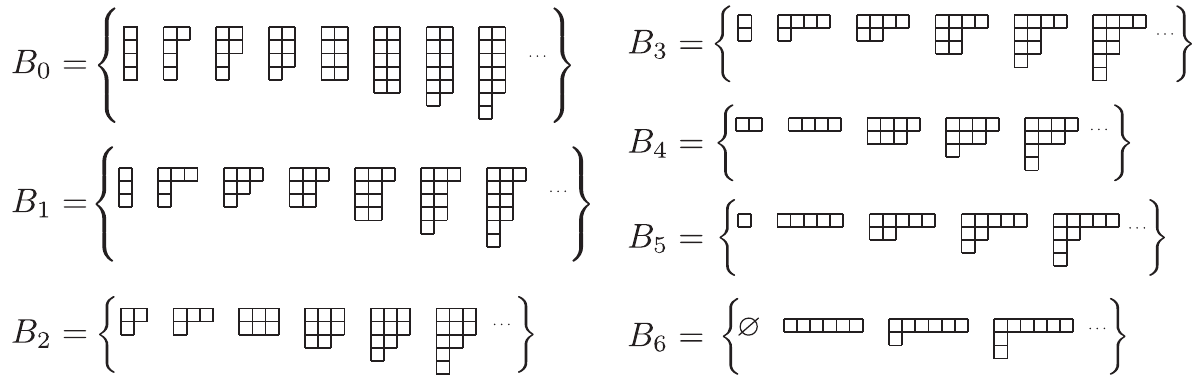}$$
~\hfill$\Dox$
\end{example}


\subsection{The functor $-\otimes L(\Box)$}\label{part1}

In this section we explain how to decompose the tensor product $L(\lambda)\otimes L(\Box)$ in $\uRep(S_T; K)$ where $\lambda$ is an arbitrary Young diagram.  
The following lemma will be our main tool in our study of decomposing tensor products.

\begin{lemma}\label{tdlemma} Suppose $\lambda, \mu, \lambda^{(1)},\ldots, \lambda^{(m)}$ are Young diagrams with the property   $L(\lambda)\otimes L(\mu)=L(\lambda^{(1)})\oplus\cdots\oplus L(\lambda^{(m)})$ in $\uRep(S_T; K)$.  Then there exists an integer $N$ such that $L_{\lambda(d)}\otimes L_{\mu(d)}=L_{\lambda^{(1)}(d)}\oplus\cdots\oplus L_{\lambda^{(m)}(d)}$ in $\Rep(S_d; F)$ whenever $d$ is an integer with $d\geq N$.
\end{lemma}

\begin{proof} By Proposition \ref{Liftprop}(4) we can find an integer $N$ such that whenever $d$ is an integer with $d\geq N$ the following two properties are satisfied: \begin{enumerate}
\item[($\star$)] $\Lift_d$ acts as the identity on all of the objects $L(\lambda), L(\mu), L(\lambda^{(1)}),\ldots, L(\lambda^{(m)})$.

\item[($\star\star$)] $\lambda(d), \mu(d), \lambda^{(1)}(d),\ldots, \lambda^{(m)}(d)$ are all Young diagrams.
\end{enumerate}
If $d\geq N$, by ($\star$) and Proposition \ref{Liftprop}(1), we have $$\Lift_d(L(\lambda)\otimes L(\mu))=L(\lambda^{(1)})\oplus\cdots\oplus L(\lambda^{(m)})$$ in $\uRep(S_T; K)$.  Hence, by ($\star$) and Proposition \ref{Liftprop}(3), $$L(\lambda)\otimes L(\mu)=L(\lambda^{(1)})\oplus\cdots\oplus L(\lambda^{(m)})$$ in $\uRep(S_d; F)$.  If we apply the tensor functor $\cat{F}$ to the equality above, the result follows from Proposition \ref{Findy} along with ($\star\star$).
\end{proof}

From Lemma \ref{tdlemma}, along with the well-known simple algorithm for decomposing $L_\lambda\otimes L_{(d-1, 1, 0,\ldots)}$ in $\Rep(S_d; F)$ (see e.g. \cite{MR0095209}), we have the following algorithm for decomposing $L(\lambda)\otimes L(\Box)$ in $\uRep(S_T; K)$.

\begin{proposition}\label{tenone}  $L(\lambda)\otimes L(\Box)=\bigoplus_{\nu}L(\nu)^{\oplus a_\nu}$  in $\uRep(S_T; K)$ where $a_\nu$ is the number of times the Young diagram $\nu$ is obtained from $\lambda$ as a result of one of the following three procedures:

(1) Add an addable\footnote{Given a Young diagram $\lambda$, an addable (resp. removable) box of $\lambda$ is a box which can be added to (resp. removed from) $\lambda$ resulting in a new Young diagram.} box to $\lambda$.

(2) Delete a removable box from $\lambda$.

(3) First delete a removable box from $\lambda$, then add an addable box to the resulting Young diagram.

\end{proposition}

\begin{example}\label{expart1}  In this example we will apply Proposition \ref{tenone} to $L(\lambda)\otimes L(\Box)$ in $\uRep(S_T; K)$ where $\lambda=(3, 1)$.  The following figure illustrates  all of the adding and removing of boxes necessary to apply Proposition \ref{tenone}.
$$\includegraphics{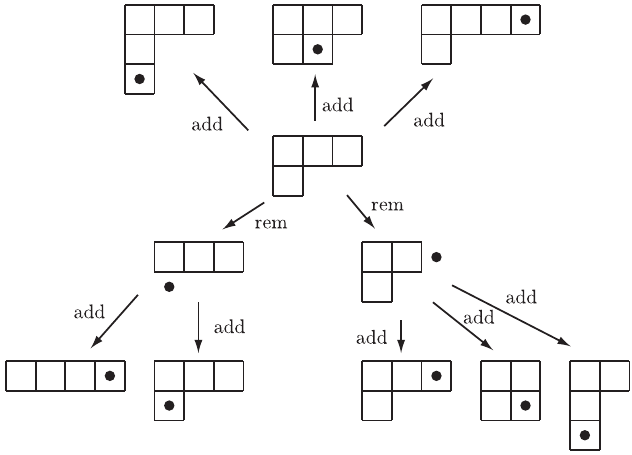}$$
Therefore, in $\uRep(S_T; K)$, $L(\lambda)\otimes L(\Box)$ decomposes as a direct sum of two copies of $L(\lambda)$ and one copy of the indecomposable object corresponding to each of the Young diagrams $(3, 1^2), (3, 2), (4, 1), (3), (2, 1), (4), (2^2), (2, 1^2)$.
 \hfill$\Dox$
\end{example}

\begin{corollary}\label{tenonecor}  Suppose $L(\lambda)\otimes L(\Box)=\bigoplus_{\nu}L(\nu)^{\oplus a_\nu}$  in $\uRep(S_T; K)$ and let $d$ be a nonnegative integer..  Then $a_\nu=1$ whenever $\nu\not=\lambda$ is such that the coordinates of $\mu_\lambda(d)-\mu_\nu(d)$ are all zero except for exactly one $1$ and one $-1$, and $a_\nu=0$ for all other $\nu\not=\lambda$.
\end{corollary}

\begin{proof} Assume $\nu\not=\lambda$ is a Young diagram and write $\mu_\lambda(d)=(\mu_0, \mu_1,\ldots)$ and $\mu_\nu(d)=(\mu_0', \mu_1',\ldots)$.   $\nu$ is obtained from $\lambda$ by adding one addable box (resp. deleting a removable box) if and only if $\mu_0-\mu_0'=|\nu|-|\lambda|$ equals $1$ (resp. $-1$), and 
there exists an $i>0$ such that $\mu_i-\mu'_i=\lambda_i-\nu_i$ is equal to $-1$ (resp. $1$) and $\mu_j-\mu'_j=\lambda_j-\nu_j=0$ for all $0<j\not=i$.  Similarly, $\nu$ is obtained from $\lambda$ by first deleting a removable box and then adding an addable box if and only if $\mu_0=\mu_0'$ and the numbers $\mu_1-\mu_1',\mu_2-\mu_2',\ldots$ are all zero except for exactly one $1$ and one $-1$.  The result now follows from Proposition \ref{tenone}.
\end{proof}

We conclude this section with a technical lemma concerning tensoring with $L(\Box)$ which will be useful in subsequent sections.  

\begin{lemma}\label{oneboxlemma} Fix a nonnegative integer $d$ and suppose $B=\{\lambda^{(0)}\subset \lambda^{(1)}\subset\cdots\}$ is a nontrivial $\stackrel{d}{\sim}$-equivalence class.  

(1)  If $B$ is minimal with respect to $\prec$, then for each integer $i>0$ there exists a Young diagram $\rho$, in a trivial $\stackrel{d}{\sim}$-equivalence class, with $L(\rho)\otimes L(\Box)=\bigoplus_{\nu}L(\nu)^{\oplus a_\nu}$ in $\uRep(S_T; K)$ where $$a_{\lambda^{(j)}}=\left\{\begin{array}{ll}
1 & \text{if }j\in\{i, i-1\},\\
0 & \text{if }j\not\in\{i, i-1\}.\\
\end{array}\right.$$

(2) If $B$ is not minimal with respect to $\prec$, then there exists a nontrivial $\stackrel{d}{\sim}$-equivalence class $B'=\{\rho^{(0)}\subset \rho^{(1)}\subset\cdots\}$  such that $B'\precneqq B$,~ and for each integer $i\geq0$, $L(\rho^{(i)})\otimes L(\Box)=\bigoplus_{\nu}L(\nu)^{\oplus a_\nu}$ and $L(\lambda^{(i)})\otimes L(\Box)=\bigoplus_{\nu}L(\nu)^{\oplus b_\nu}$ in $\uRep(S_T; K)$ where $$a_{\lambda^{(j)}}=b_{\rho^{(j)}}=\left\{\begin{array}{ll}
1 & \text{if }j=i,\\
0 & \text{if }j\not=i.\\
\end{array}\right.$$
\end{lemma}

\begin{proof} (1) Assume $B$ is the minimal $\stackrel{d}{\sim}$-equivalence class.  Then \begin{equation}\label{miniblock}\lambda^{(i)}=\left\{\begin{array}{ll}
(2^i, 1^{d-i-1}) & \text{if }0\leq i<d,\\
(2^d, 1^{i-d}) & \text{if }i\geq d.\\
\end{array}\right.\end{equation}  Now fix $i>0$ and set $$\rho:=\left\{\begin{array}{ll}
(1^d) & \text{if }i=1,\\
(3, 2^{i-2}, 1^{d-i})&  \text{if } 1<i<d,\\
(2^{d-1}, 1) & \text{if }i=d,\\
(3, 2^{d-1}, 1^{i-d-1})  & \text{if } i> d.
\end{array}\right.$$  If we set $\mu_\rho(d)=(\mu_0, \mu_1,\ldots)$, then it is easy to check that $$\mu_0=\left\{\begin{array}{ll}
\mu_i & \text{if }1\leq i\leq d,\\
\mu_{i+1} & \text{if }i> d.\\
\end{array}\right.$$ Hence $\rho$ is in a trivial $\stackrel{d}{\sim}$-equivalence class (see Corollary \ref{gensim}).   Finally, comparing $\rho$ to the Young diagrams in (\ref{miniblock}) and using Proposition \ref{tenone}, it is easy to check that $\rho$ satisfies part (1).

(2)  Suppose $B$ is not minimal so that $\lambda^{(0)}(d)=(\lambda_0, \lambda_1,\ldots)$ is a Young diagram of size $d$ distinct from $(1^d)$.  Then $\lambda^{(0)}(d)$ has a row with more than one box.  Let $m$ be the maximal integer with $\lambda_m>1$, and let $k$ be the minimal positive integer with $\lambda_k=0$.  Let $\rho^{(0)}$ be the Young diagram with $\rho^{(0)}(d)$ obtained from $\lambda^{(0)}(d)$ by deleting one box from row $m$ and adding one box to row $k$
$$\includegraphics{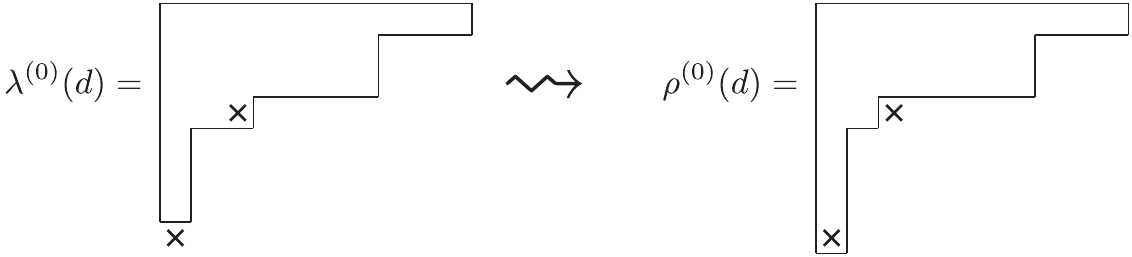}$$ 
Since $\rho^{(0)}(d)$ is a Young diagram of size $d$ with $\rho^{(0)}(d)\prec\lambda^{(0)}(d)$, $\rho^{(0)}$ is the minimal element in a nontrivial $\stackrel{d}{\sim}$-equivalence class $B'=\{\rho^{(0)}\subset\rho^{(1)}\subset\cdots\}$ which satisfies  $B'\precneqq B$ (see Corollary \ref{mindiag}).  
By Corollary \ref{tenonecor} it suffices to show that the coordinates of $\mu_{\lambda^{(i)}}(d)-\mu_{\rho^{(j)}}(d)$ are all zero except for one $1$ and one $-1$ if and only if $i=j$.

Write $\mu_{\lambda^{(0)}}(d)=(\mu_0, \mu_1,\ldots)$ and $\mu_{\rho^{(0)}}(d)=(\mu'_0, \mu'_1,\ldots)$.  For an integer $i\geq0$, let $\tau_i:\Z_{\geq 0}\to\Z_{\geq 0}$ be the bijection mapping $0\mapsto i$, $j\mapsto j-1$ ($0<j\leq i$), and $j\mapsto j$ ($j>i$).  Then, by Proposition \ref{mindiagprop}, $\mu_{\lambda^{(i)}}(d)=(\mu_{\tau_i(0)}, \mu_{\tau_i(1)},\ldots)$ and $\mu_{\rho^{(j)}}(d)=(\mu'_{\tau_j(0)}, \mu'_{\tau_j(1)},\ldots)$ for all $i, j\geq 0$.  
Now suppose $i, j\geq 0$ are such that the coordinates of $\mu_{\lambda^{(i)}}(d)-\mu_{\rho^{(j)}}(d)$ are all zero except for one $1$ and one $-1$.  Equivalently, setting $\tau=\tau_i^{-1}\tau_j$, there exist nonnegative integers $m', k'$ such that \begin{equation}\label{mutau}\mu_l-\mu_{\tau(l)}'=\left\{\begin{array}{ll}
1 & \text{if $l=m'$,}\\
-1& \text{if $l=k'$,}\\
0 & \text{otherwise.}
\end{array}\right.\end{equation}  
By the construction of $\rho^{(0)}(d)$ we have $$\mu_l-\mu_l'=\left\{\begin{array}{ll}
1 & \text{if $l=m$,}\\
-1& \text{if $l=k$,}\\
0 & \text{otherwise.}
\end{array}\right.$$
Since $\mu_0>\mu_1>\cdots$ (resp. $\mu'_0>\mu'_1>\cdots$), neither $\mu_m$ nor $\mu_k$ (resp. $\mu'_m$ nor $\mu'_k$) is equal to $\mu_l=\mu_l'$ whenever $l\not=m, k$.  Moreover, since $k>m$ we have $\mu_m-\mu_k'=\mu_m'-\mu_k'+1>1$ and $\mu_k-\mu_m'=\mu_k'-\mu_m'-1<-1$.  Hence (\ref{mutau}) holds if and only if $m=m'$, $k=k'$, and $\tau$ is the identity map.  However, $\tau=\tau_i^{-1}\tau_j$ is the identity map if and only if $i=j$.  This completes the proof.
\end{proof}


\subsection{Proof of result on blocks}\label{secblocks} In this section we give a proof of Theorem \ref{blocksinSt}.  To start, given a Young diagram $\lambda$ we will describe how $\Lift_t(L(\lambda))$ decomposes as a sum of indecomposable objects in $\uRep(S_T; K)$ (see Lemma \ref{liftdec}).  This description, along with Lemma \ref{equalxi}  will be used to prove Theorem \ref{blocksinSt}.  We begin with the following proposition.

\begin{proposition}\label{liftsim} Suppose $\lambda, \lambda^{(1)},\ldots, \lambda^{(m)}$ are Young diagrams with the property  that $\Lift_t(L(\lambda))=L(\lambda^{(1)})\oplus\cdots\oplus L(\lambda^{(m)})$ in $\uRep(S_T; K)$.  Then $\lambda^{(i)}\stackrel{t}{\sim}\lambda$ for all $i=1,\ldots, m$.  
\end{proposition}

\begin{proof} Fix a positive integer $k$ such that $\lambda^{(i)}_k=\lambda_k=0$ for all $i=1,\ldots, m$.  By Lemma \ref{equalxi}(2), it suffices to show $\xi^{\lambda^{(i)}(t)}_{r, k}=\xi^{\lambda(t)}_{r, k}$ for all $r>0$ and $i=1,\ldots, m$.  Fix a positive integer $r$, set $n=|\lambda|$, $e=e_\lambda$ and let $\ep\in KP_n(T)$ be a lift of $e$.  Let $A(x)\in K[x]$ be the minimal monic polynomial with $A(\omega_n^r(T))\ep=0$.  If $\ep=\ep_1+\cdots+\ep_{m'}$ is an orthogonal decomposition of $\ep$ into primitive idempotents, then $m'=m$ and (after reordering) $L(\lambda^{(i)})=([n], \ep_i)$ for all $i=1,\ldots, m$.  Hence $A(x)$ is the product of linear terms of the form $x-\xi^{\lambda^{(i)}(T)}_{r, k}$ (see Proposition \ref{minpoly}).  Let $B(x), C(x)\in K[x]$ be the unique monic polynomials with $A=BC$ such that $B(x)$ (resp. $C(x)$) is the product of linear terms of the form $x-\xi^{\lambda^{(i)}(T)}_{r, k}$ with $\xi^{\lambda^{(i)}(t)}_{r, k}$ equal to (resp. not equal to) $\xi^{\lambda(t)}_{r, k}$.  Suppose for a contradiction that $C(x)\not=1$.  Since $B(x)$ and $C(x)$ are relatively prime polynomials of positive degree, there exist nonzero polynomials $G(x), H(x)\in K[x]$ with $\deg G(x)<\deg C(x)$,  $\deg H(x)<\deg B(x)$ and \begin{equation}\label{gbhc}G(x)B(x)+H(x)C(x)=1.\end{equation}  Let $N$ be the minimal integer so that all coefficients of both $G'(x):=(T-t)^NG(x)$ and $H'(x):=(T-t)^NH(x)$ lie  in $F[[T-t]]$.  Then from equation (\ref{gbhc}) we have \begin{equation}\label{gbhcprime}G'(x)B(x)+H'(x)C(x)=(T-t)^N.\end{equation}
Let $b(x), c(x), g(x), h(x)\in F[x]$ be the polynomials obtained by evaluating $T=t$ in the polynomials $B(x), C(x), G'(x), H'(x)$ respectively.  If $N>0$, then equation (\ref{gbhcprime}) implies $b(x)$ (resp. $c(x)$) divides $h(x)$ (resp. $g(x)$).  On the other hand, $\deg b(x)=\deg B(x)>\deg H(x)\geq\deg h(x)$.  Similarly $\deg c(x)>\deg g(x)$.  Thus $h(x)=g(x)=0$ which contradicts the minimality of $N$.  Hence $N=0$ which implies the coefficients of $H(x)=H'(x)$ and $G(x)=G'(x)$ are in $F[[T-t]]$. 
Thus evaluating $T=t$, $x=\omega_n^r(t)$ in equation (\ref{gbhc}) and multiplying by $e$ gives \begin{equation}\label{gbhcomega}g(\omega_n^r(t))b(\omega_n^r(t))e+h(\omega_n^r(t))c(\omega_n^r(t))e=e.
\end{equation}
Since $\omega_n^r(t)$ is in the center of $FP_n(t)$ (Proposition \ref{centralomega}(2)) and $b(\omega_n^r(t))c(\omega_n^r(t))e=0$, the two summands on the left side of equation (\ref{gbhcomega}) are idempotents.  As $e$ is assumed to be primitive, either $g(\omega_n^r(t))b(\omega_n^r(t))e=0$ or $h(\omega_n^r(t))c(\omega_n^r(t))e=0$.  Since $G(\omega_n^r(T))B(\omega_n^r(T))\ep$ and $H(\omega_n^r(T))C(\omega_n^r(T))\ep$ are lifts of $g(\omega_n^r(t))b(\omega_n^r(t))e$ and $h(\omega_n^r(t))c(\omega_n^r(t))e$ respectively, it follows that 
 either $G(\omega_n^r(T))B(\omega_n^r(T))\ep=0$ or $H(\omega_n^r(T))C(\omega_n^r(T))\ep=0$ (see Theorem \ref{appenlift}) which contradicts the minimality of $A(x)$.  This completes the proof.
\end{proof}

Among other things, Proposition \ref{liftsim} implies that $\Lift_t(L(\lambda))=L(\lambda)$ whenever $\lambda$ is in a trivial $\stackrel{t}{\sim}$-equivalence class. The following lemma describes $\Lift_t(L(\lambda))$ for arbitrary $\lambda$.

\begin{lemma}\label{liftdec} (1) Suppose $t\in F$ and  $\lambda$ is a Young diagram which is either in a trivial $\stackrel{t}{\sim}$-equivalence class or the minimal element of a nontrivial $\stackrel{t}{\sim}$-equivalence class.  Then $\Lift_t(L(\lambda))=L(\lambda)$.

(2)  Suppose $d$ is a nonnegative integer and $B=\{\lambda^{(0)}\subset\lambda^{(1)}\subset\cdots\}$ is a nontrivial $\stackrel{d}{\sim}$-equivalence class.  Then $\Lift_d(L(\lambda^{(i)}))=L(\lambda^{(i)})\oplus L(\lambda^{(i-1)})$ for all $i>0$.
\end{lemma}

\begin{proof} Part (1) follows from Propositions \ref{Liftprop}(3) and \ref{liftsim}.  To prove part (2), assume $\Lift_d(L(\lambda^{(i)}))$ is indecomposable.  Then it follows from  Proposition \ref{Liftprop}(3) that  $\Lift_d(L(\lambda^{(i)}))=L(\lambda^{(i)})$.  Hence, by Proposition \ref{Liftprop}(2) and Proposition \ref{gendimen}, $\dim_{\uRep(S_d; F)}L(\lambda^{(i)})=P_{\lambda^{(i)}}(d)$, which is nonzero by Proposition \ref{zero}.  Therefore $\cat{F}(L(\lambda^{(i)}))$ is nonzero in $\Rep(S_d; F)$, which is only possible if $i=0$ (see Propositions \ref{Findy} and  \ref{mindiag}).  Hence $\Lift_d(L(\lambda^{(i)}))$ is decomposable whenever $i>0$.  We now proceed by induction on $\prec$.  

For our base case, assume that $B$ is the minimal $\stackrel{d}{\sim}$-equivalence class.  Fix $i>0$ and let $\rho$ be as in as in Lemma \ref{oneboxlemma}(1).  Since $\rho$ is in a trivial $\stackrel{d}{\sim}$-equivalence class $\Lift_d(L(\rho))=L(\rho)$ (by part (1)).  Hence, by Example \ref{lift1}(1), Proposition \ref{Liftprop}(1), and Lemma \ref{oneboxlemma}(1), $$\Lift_d(L(\rho)\otimes([1], \id_1))=L(\rho)\otimes(L(\Box)\oplus L(\varnothing))=L(\lambda^{(i)})\oplus L(\lambda^{(i-1)})\oplus A$$ in $\uRep(S_T; K)$ where $A$ is an object with no indecomposable summands of the form $L(\lambda^{(j)})$.  By Propositions \ref{Liftprop}(3) and \ref{liftsim}, $L(\rho)\otimes([1], \id_1)$ must have $L(\lambda^{(i)})$ as an indecomposable summand in $\uRep(S_d; F)$.  Moreover, since $\Lift_d(L(\lambda^{(i)}))$ is decomposable, it must be the case that $\Lift_d(L(\lambda^{(i)}))=L(\lambda^{(i)})\oplus L(\lambda^{(i-1)})$.  

Now assume $B$ is not minimal and let $B'=\{\rho^{(0)}\subset\rho^{(1)}\subset\cdots\}$ be as in Lemma \ref{oneboxlemma}(2).  Since $B'\precneqq B$, by induction we have $\Lift_d(L(\rho^{(i)}))=L(\rho^{(i)})\oplus L(\rho^{(i-1)})$ whenever $i>0$.   Hence, by Example \ref{lift1}(1), Proposition \ref{Liftprop}(1) and Lemma \ref{oneboxlemma}(2), $$\begin{array}{rcl}\Lift_d(L(\rho^{(i)})\otimes([1], \id_1)) & = &(L(\rho^{(i)})\oplus L(\rho^{(i-1)}))\otimes(L(\Box)\oplus L(\varnothing))\\
& = & L(\lambda^{(i)})\oplus L(\lambda^{(i-1)})\oplus A\end{array}$$ in $\uRep(S_T; K)$ where $A$ is an object with no indecomposable summands of the form $L(\lambda^{(i)})$.  As in the base case, this implies the desired result.
\end{proof}

\begin{remark} Lemma \ref{liftdec}(2) implies that $P_{\lambda^{(i)}}(d)=-P_{\lambda^ 
{(i-1)}}(d)$
for $i>0$ whence $P_{\lambda^{(i)}}(d)=(-1)^i P_{\lambda^{(0)}}(d)$. This  
identity was
observed by Deligne, see \cite[\S 7.5]{Del07}.
\end{remark}

We are ready to prove the main result section \ref{blocks}, which we restate now.

\medskip

\noindent{\bf Theorem \ref{blocksinSt}.} \emph{$L(\lambda)$ and $L(\lambda')$ are in the same block of $\uRep(S_t; F)$ if and only if $\lambda\stackrel{t}{\sim}\lambda'$.}

\begin{proof}   If $L(\lambda)$ and $L(\lambda')$ are in the same block of $\uRep(S_t; F)$, then Lemma \ref{equalxi} implies $\lambda\stackrel{t}{\sim}\lambda'$.
To prove the converse, by Corollary \ref{gensim}(1) we may assume $t=d$ is a nonnegative integer.  Suppose $\{\lambda^{(0)}\subset \lambda^{(1)}\subset\cdots\}$ is a nontrivial $\stackrel{d}{\sim}$-equivalence class and fix $i> 0$. By Proposition \ref{Liftprop}(5) $\dim_F\Hom_{\uRep(S_d; F)}(L(\lambda^{(i-1)}), L(\lambda^{(i)}))$ is equal to $\dim_K\Hom_{\uRep(S_T; K)}(\Lift_d(L(\lambda^{(i-1)})), \Lift_d(L(\lambda^{(i)})))$, which is nonzero by Lemma \ref{liftdec}.  The result follows.
\end{proof}

We close this section by examining the dimensions of the Hom spaces between indecomposable objects in $\uRep(S_t; F)$.

\begin{proposition}\label{dimstuff} (1) $\dim_F\operatorname{End}_{\uRep(S_t; F)}(L(\lambda))=1$ whenever $\lambda$ is in a trivial $\stackrel{t}{\sim}$-equivalence class.  In particular, the block corresponding to a trivial $\stackrel{t}{\sim}$-equivalence class is trivial.

(2) Given a nontrivial block 
$\{\lambda^{(0)}\subset\lambda^{(1)}\subset\cdots\}$ in $\uRep(S_t; F)$, $$\dim_F\Hom_{\uRep(S_t; F)}(L(\lambda^{(i)}), L(\lambda^{(j)}))=\left\{\begin{array}{ll}
2& \text{if $i=j>0$,}\\
0& \text{if $|i-j|>2$,}\\
1& \text{otherwise.}
\end{array}\right.$$
\end{proposition}

\begin{proof} By Lemma \ref{liftdec} and Proposition \ref{Liftprop}(5), it suffices to prove $$\dim_K\Hom_{\uRep(S_T; K)}(L(\lambda), L(\lambda'))=\left\{\begin{array}{ll}
1& \text{if $\lambda=\lambda'$,}\\
0& \text{if $\lambda\not=\lambda'$.}
\end{array}\right.$$  By Proposition \ref{fieldex} it suffices to consider the case when $K$ is algebraically closed.  As $\uRep(S_T; K)$ is semisimple (Corollary \ref{ssT}), the result follows from Schur's lemma.
\end{proof}

\begin{corollary}\label{ss} $\uRep(S_t; F)$ is semisimple if and only if $t$ is not a nonnegative integer.
\end{corollary}

\begin{proof} This follows from Proposition \ref{dimstuff}(1) along with corollaries \ref{gensim}(1) and  \ref{mindiag}.
\end{proof}


\section{Description of a non-semisimple block} \label{blockquiver}

In this section we give a complete description of nontrivial blocks in $\uRep(S_d; F)$ for all $d\in\Z_{\geq0}$.  In particular, we prove that all nontrivial blocks are equivalent as additive categories.  We begin by describing the nontrivial block in $\uRep(S_0; F)$.  Next, we show that for fixed $d\in\Z_{\geq0}$, the nontrivial blocks in $\uRep(S_d; F)$ are all equivalent.  We then state a conjecture which would allow us to compare blocks in $\uRep(S_d; F)$ with those in $\uRep(S_{d-1}; F)$ using a ``restriction functor."  Finally, we use Martin's results on the partition algebras to give a complete description of nontrivial blocks.


\subsection{The nontrivial block in $\uRep(S_0; F)$}  In this section we give a complete description of the one nontrivial block in $\uRep(S_0; F)$.  In this particular case the constructions of all idempotents are easy enough that we are able to fully describe the block by brute force computations.  
We expect this method is too computationally complicated in other cases.

Throughout this section we consider the group algebra of the symmetric group $FS_n$ as a subalgebra of the partition algebra $FP_n(0)$ (see Remark \ref{Snembedding}).  With this in mind, we have the following idempotents:\begin{equation*} s_n:=\frac{1}{n!}\sum_{\sigma\in S_n}sgn(\sigma)\sigma\in FP_n(0).
\end{equation*}

\begin{proposition}\label{sgn} In $\uRep(S_0; F)$, $([n], s_n)\cong L((1^n))$ for all $n\geq 0$.
\end{proposition}

\begin{proof} The proposition is certainly true when $n=0$, so we assume $n>0$.  Since the projection $FP_n(0)\twoheadrightarrow FS_n$ maps $s_n\mapsto s_n$, and in $FS_n$ the idempotent $s_n$ is primitive corresponding to $(1^n)$, we know any orthogonal decomposition of $s_n$ into primitive idempotents in $FP_n(0)$ must contain a summand corresponding to $L((1^n))$ in $\uRep(S_0; F)$.  Thus, by Example \ref{3class}(1) and Proposition \ref{dimstuff} it suffices to show $\dim_F s_nFP_n(0)s_n=2$.

Suppose $\mu\in P_{n,n}$ satisfies $\tau\mu=\pi$ for some transposition $\tau\in S_n\subset P_{n,n}$.  Then $s_n\mu s_n=s_n\tau\mu s_n=-s_n\pi s_n$, which implies $s_n\mu s_n=0$.  Similarly, if $\mu\tau=\mu$ for some transposition $\tau$, then $s_n\mu s_n=0$.  The only elements of $\mu\in P_{n,n}$ such that $\tau\mu\not=\mu$ and $\mu\tau\not=\mu$ for all transpositions $\tau\in S_n$ are either elements of $S_n$ or of the form $\sigma x\sigma'$ for some $\sigma, \sigma'\in S_n$ with $x=\id_{n-1}\otimes\pi$ where $\pi\in P_{1,1}$ is as in the proof of Theorem \ref{idemP}.  If $\mu\in S_n$, then $s_n\mu s_n=\pm s_n$.  If $\mu=\sigma x\sigma'$ for some $\sigma, \sigma'\in S_n$, then $s_n\mu s_n=\pm s_n x s_n$.  Hence, $s_n$ and $s_n x s_n$ span $s_nFP_n(0)s_n$.
\end{proof}

As a consequence of Example \ref{3class}(1) and Proposition \ref{sgn}, describing the nontrivial block in $\uRep(S_0; F)$ amounts to describing morphisms among the objects $([n], s_n)$ for $n\geq0$.  To describe such morphisms, let $x^1_0$ denote the unique element of $P_{1,0}$ and let 
$x_n^{n+1}=\id_n\otimes x^1_0$ for all $n>0$.  Finally, let $x^n_{n+1}=(x^{n+1}_n)^\vee$ for all $n\geq 0$ as pictured below.  $$\includegraphics{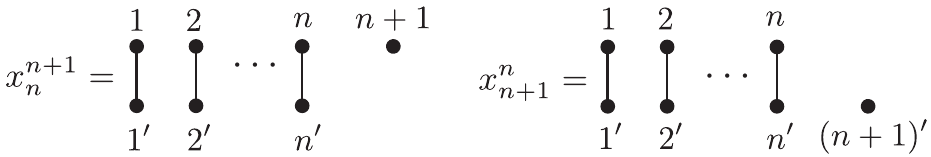}$$ The following lemma records a couple properties useful in forthcoming calculations.

\begin{lemma}\label{calcs} The following identities hold in $\uRep(S_0; F)$:

(1) $s_n x_n^{n+1}s_{n+1}=x_n^{n+1}s_{n+1}$ for all $n\geq 0$.

(2) $-ns_nx_n^{n-1}x_{n-1}^ns_n=(n+1)x_n^{n+1}s_{n+1}x^n_{n+1}$ for all $n> 0$.

(3) $x_{n-1}^nx_n^{n+1}s_{n+1}=0$ for all $n>0$.
\end{lemma}

\begin{proof} To prove part (1), let $s_n'=s_n\otimes\id_1\in FP_{n+1}(0)$ for all $n\geq 0$.  Then $s_n x_n^{n+1}s_{n+1}=x_n^{n+1}s_n's_{n+1}=x_n^{n+1}s_{n+1}$.

To prove part (2), notice that for $\sigma\in S_{n+1}$ there are exactly $(n-1)!$ pairs $(\tau_1, \tau_2)$ with $\tau_1, \tau_2\in S_n$ such that $\tau_1x_n^{n-1}x_{n-1}^n\tau_2=x_n^{n+1}\sigma x_{n+1}^n$.  Moreover, for such a pair $sgn(\tau_1\tau_2)=-sgn(\sigma)$.  Conversely, given any $\tau_1, \tau_2\in S_n$, either $\tau_1x_n^{n-1}x_{n-1}^n\tau_2=0$ or there exists a unique $\sigma\in S_{n+1}$ with $\tau_1x_n^{n-1}x_{n-1}^n\tau_2=x_n^{n+1}\sigma x_{n+1}^n$.  Therefore, $-(n!s_n)x_n^{n-1}x_{n-1}^n(n!s_n)=(n-1)!x_n^{n+1}(n+1)!s_{n+1}x^n_{n+1}$ which is equivalent to the identity in part (2).

To prove part (3), let $\tau\in S_{n+1}$ denote the transposition $n\leftrightarrow n+1$.  Then 
$x_{n-1}^nx_n^{n+1}s_{n+1}=(x_{n-1}^nx_n^{n+1}\tau) s_{n+1}=x_{n-1}^nx_n^{n+1}(\tau s_{n+1})=-x_{n-1}^nx_n^{n+1}s_{n+1}$.  The result follows.
\end{proof}

Next, define the following morphisms:
\begin{equation*} \alpha_n:=(-1)^n(n+1)!s_{n+1}x^n_{n+1}s_{n},\qquad\beta_n:=\frac{1}{n!}s_nx^{n+1}_ns_{n+1}\quad (n\geq0),\end{equation*}
\begin{equation*}\gamma_n:=(-1)^nns_nx_n^{n-1}x_{n-1}^ns_n\quad(n>0).\end{equation*}

The next lemma contains all calculations needed to describe the nontrivial block in $\uRep(S_0; F)$.

\begin{lemma}\label{calculations}The following equations hold in $\uRep(S_0; F)$:

(1) $\alpha_n\not=0$ for $n\geq0$.

(2) $\beta_n\not=0$ for $n\geq0$.

(3) $\gamma_n\not=0$ for $n>0$.

(4) $\gamma_n^2=0$ for $n>0$, hence $\gamma_n$ is not a scalar multiple of $s_n$.

(5) $\beta_0\alpha_0=0$.

(6) $\beta_n\alpha_n=\gamma_n$ for $n>0$.

(7) $\alpha_{n-1}\beta_{n-1}=\gamma_n$ for $n>0$.

(8) $\alpha_n\alpha_{n-1}=0$ for $n>0$.

(9) $\beta_{n-1}\beta_n=0$ for $n>0$.

\end{lemma}

\begin{proof} Up to a nonzero scalar multiple,  $\alpha_n$ and $\beta_n^\vee$ are equal.  Hence, parts (1) and (8) will follow from parts (2) and (9) respectively.  

(2) Write $\beta_n=\sum_{\pi\in P_{n+1, n}}b_\pi\pi$. Then $b_{x_n^{n+1}}=\frac{1}{n!(n+1)!}$.

(3) Write $\gamma_n=\sum_{\pi\in P_{n+1, n}}c_\pi\pi$. Then $c_{x_n^{n-1}x_{n-1}^{n}}=\frac{(-1)^n}{n!}$.

(4) $\gamma_n^2=-n(n+1)s_nx_n^{n-1}x_{n-1}^nx_n^{n+1}s_{n+1}x_{n+1}^n=0$, (Lemma \ref{calcs}(2)\&(3)).  

(5) $\beta_0\alpha_0=-x_0^1x_1^0=0$ in $\uRep(S_0; F)$.

(6) $\beta_n\alpha_n=(-1)^{n+1}(n+1)s_nx_n^{n+1}s_{n+1}x_{n+1}^ns_n=\gamma_n$, (Lemma \ref{calcs}(2)).

(7) $\alpha_{n-1}\beta_{n-1}=(-1)^{n}ns_nx_n^{n-1}s_{n-1}x_{n-1}^ns_n=\gamma_n$, (Lemma \ref{calcs}(1)).

(9) $\beta_{n-1}\beta_n=\frac{1}{(n-1)!n!}s_{n-1}x_{n-1}^ns_nx_n^{n+1}s_{n+1}$, which (by Lemma \ref{calcs}(1)) is equal to $\frac{1}{(n-1)!n!}s_{n-1}x_{n-1}^nx_n^{n+1}s_{n+1}$, which (by Lemma \ref{calcs}(3)) is equal to zero.
\end{proof}

The following theorem describes the nontrivial block in $\uRep(S_0; F)$.

\begin{theorem}\label{zeroblock} Let $L_n:=L((1^n))$ for all $n\geq 0$.  The nontrivial block in $\uRep(S_0; F)$ has the following associated quiver  $$\includegraphics{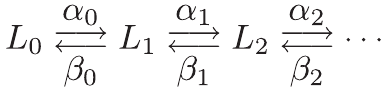}$$ 
with relations: $\beta_0\alpha_0=0$, $\alpha_{n}\alpha_{n-1}=0$, $\beta_{n-1}\beta_{n}=0$,  $\beta_n\alpha_n=\alpha_{n-1}\beta_{n-1}$ for $n>0$.
\end{theorem}

\begin{proof} The indecomposable objects in the block follow from Example \ref{3class}(1).  The fact that the arrows generate all the morphisms between the indecomposable objects follows from Lemma \ref{dimstuff} and Lemma \ref{calculations}(1)-(4)\&(6).  By Lemma \ref{calculations}(5)-(9), the relations hold.  Finally, the relations imply that the only nonzero compositions of $\alpha$'s and $\beta$'s are of the form $\alpha_n~ (n\geq 0), \beta_n~ (n\geq 0)$, and $\beta_n\alpha_n ~(n> 0)$.  By Lemma \ref{calculations}(4)\&(6), $\beta_n\alpha_n $ and $s_n=\id_{L_n}$ are linearly independent.  Hence the relations are exhaustive.
\end{proof}

\begin{remark} The block described in Theorem \ref{zeroblock} is equivalent to the nontrivial blocks in the 
{\em Temperley-Lieb category}. This can be deduced from the results in \cite[\S 10]{Wes}.
\end{remark}


\subsection{Comparison of non-semisimple blocks in $\uRep(S_d; F)$.}  In this section we show that for fixed $d\in\Z_{\geq0}$, the non-semisimple blocks in $\uRep(S_d; F)$ are all equivalent as additive categories.  Thereafter we conjecture that a restriction functor induces an equivalence of categories between certain non-semisimple blocks in $\uRep(S_d; F)$ and $\uRep(S_{d-1}; F)$.  
First, let us fix some notation.  Given a block $\cat{B}$ in $\uRep(S_d; F)$, let $\Inc_\cat{B}: \cat{B}\to\uRep(S_d; F)$ and $\Proj_\cat{B}: \uRep(S_d; F)\to\cat{B}$ denote the inclusion and projection functors respectively.  

\begin{proposition}\label{blocksfixd} Suppose $\cat{B}$ is a nontrivial block in $\uRep(S_d; F)$ which is not minimal with respect to $\prec$.  There exists a block $\cat{B}'$ in $\uRep(S_d; F)$ with $\cat{B}'\precneqq\cat{B}$ such that $\Proj_{\cat{B}'}\circ(-\otimes L(\Box))\circ\Inc_{\cat{B}}:\cat{B}\to\cat{B}'$ is an equivalence of additive categories.  Hence, all nontrivial blocks in $\uRep(S_d; F)$ are all equivalent as additive categories.
\end{proposition}

\begin{proof} First, notice $d\not=0$ as we are assuming a nontrivial block exists in $\uRep(S_d; F)$.  Hence, $\Lift_d(L(\Box))=L(\Box)$ (see Example \ref{lift1}(1)).  Let $B=\{\lambda^{(0)}\subset\lambda^{(1)}\subset\cdots\}$ denote the set of Young diagrams corresponding to indecomposable objects in $\cat{B}$. Now, let $B'=\{\rho^{(0)}\subset\rho^{(1)}\subset\cdots\}$  be the set of Young diagrams given by Lemma \ref{oneboxlemma}(2), and $\cat{B}'$ the corresponding block.  By Lemma \ref{liftdec} we have the following $$\begin{array}{rcll}
\Lift_d(L(\lambda^{(0)})) & = & L(\lambda^{(0)}),\\
\Lift_d(L(\lambda^{(i)}))&=&L(\lambda^{(i)})\oplus L(\lambda^{(i-1)})&(i>0),\\
\Lift_d(L(\rho^{(0)}))&=&L(\rho^{(0)}),\\
\Lift_d(L(\rho^{(i)}))&=&L(\rho^{(i)})\oplus L(\rho^{(i-1)})& (i>0).\\
\end{array}$$
Hence, by Proposition \ref{Liftprop}(1) and Lemma \ref{oneboxlemma}(2), 
$$\begin{array}{rrcll}
\Lift_d(L(\lambda^{(i)})\otimes L(\Box))&=&\Lift_d(L(\lambda^{(i)}))\otimes L(\Box)&=&\Lift_d(L(\rho^{(i)}))\oplus A,\\
\Lift_d(L(\rho^{(i)})\otimes L(\Box))&=&\Lift_d(L(\rho^{(i)}))\otimes L(\Box)&=&\Lift_d(L(\lambda^{(i)}))\oplus B,\\
\end{array}$$ in $\uRep(S_T; K)$ for all $i\geq0$, where $A$ (resp. $B$) is an object in $\uRep(S_T; K)$ with no indecomposable summands of the form $L(\rho^{(j)})$ (resp. $L(\lambda^{(j)})$).  Thus, by Propositions \ref{Liftprop}(3) and \ref{liftsim}, 
$L(\lambda^{(i)})\otimes L(\Box)=L(\rho^{(i)})\oplus A'$ and $L(\rho^{(i)})\otimes L(\Box)=L(\lambda^{(i)})\oplus B'$  in $\uRep(S_d; F)$ for all $i\geq 0$, where $A'$ (resp. $B'$) is an object in $\uRep(S_d; F)$ with no indecomposable summands of the form $L(\rho^{(j)})$ (resp. $L(\lambda^{(j)})$).  Hence, $\Proj_{\cat{B}'}\circ(-\otimes L(\Box))\circ\Inc_{\cat{B}}$  and $\Proj_\cat{B}\circ(-\otimes L(\Box))\circ\Inc_{\cat{B}'}$ are inverse to one another on objects.  Moreover, since $(-\otimes L(\Box))$ is self-adjoint and $\Proj_\cat{B}$ is both right and left adjoint to $\Inc_{\cat{B}}$, it follows that $\Proj_\cat{B}\circ(-\otimes L(\Box))\circ\Inc_{\cat{B}'}$ is adjoint to $\Proj_{\cat{B}'}\circ(-\otimes L(\Box))\circ\Inc_{\cat{B}}$.  The result follows.
\end{proof}


Next, we use the universal property of $\uRep(S_t; F)$ (see \cite[8.3]{Del07}) to define a restriction functor.

\begin{definition} Let $\uRes_{S_{t-1}}^{S_t}:\uRep(S_t; F)\to\uRep(S_{t-1}; F)$ denote the functor given by the universal property of $\uRep(S_t; F)$ which sends $([1], \id_1)\mapsto([1], \id_1)\oplus([0], \id_0)$.
\end{definition}



\begin{conjecture}\label{conjecture} For $d\in\Z_{\geq0}$, let $\cat{B}_d$ denote the nontrivial block in $\uRep(S_d; F)$ containing the object $L(\varnothing)$.  Then the functor $\Proj_{\cat{B}_d}\circ\uRes_{S_d}^{S_{d+1}}\circ\Inc_{\cat{B}_{d+1}}$ induces an equivalence of additive categories $\cat{B}_{d+1}\cong\cat{B}_{d}$.
\end{conjecture}

\begin{remark}  
It is not hard to show that $\Proj_{\cat{B}_d}\circ\uRes_{S_d}^{S_{d+1}}\circ\Inc_{\cat{B}_{d+1}}$ is bijective on objects.  Hence, to prove Conjecture \ref{conjecture} it suffices to show $\Proj_{\cat{B}_d}\circ\uRes_{S_d}^{S_{d+1}}\circ\Inc_{\cat{B}_{d+1}}$ is either full or faithful.  
\end{remark}


\subsection{Description of blocks via Martin}  In this section we give a general description of the nontrivial blocks based on the results of Martin.  We start by reviewing the main result in \cite{MR1399030}.

Assume $d\not=0$ and let $\lambda^{(0)}\subset\lambda^{(1)}\subset\cdots$ denote the Young diagrams associated to a fixed nontrivial block in $\uRep(S_d; F)$.  For each $m\geq|\lambda^{(n)}|$, let $E_m^{(n)}$ denote the simple $FP_m(d)$-module associated to $\lambda^{(n)}$ (see Theorem \ref{idemP}(1)), and let $P_m^{(n)}$ denote its projective cover.  According to \cite[Proposition 9]{MR1399030}, these modules have Loewy structure
\begin{equation*}\label{loewy}\begin{array}{l}
P_m^{(0)}=\begin{array}{c}
 E^{(0)}_m \\
E^{(1)}_m
\end{array}\quad(m\geq|\lambda^{(1)}|),\\
\\
P_m^{(n)}=\begin{array}{rcl}
& E^{(n)}_m &\\
E^{(n-1)}_m & & E^{(n+1)}_m\\
& E^{(n)}_m &
\end{array}\quad(n>0, m\geq|\lambda^{(n+1)}|).
\end{array}
\end{equation*}  
In other words, for $m>0$, $P_m^{(0)}$ has a maximal simple submodule $B_0\cong E_m^{(1)}$  with $P_m^{(0)}/B_0\cong E^{(0)}_m$.  Moreover, 
for each $m\geq|\lambda^{(n+1)}|$ there exists a chain of submodules 
$$\includegraphics{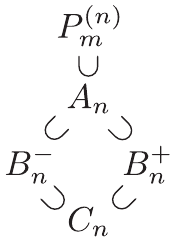}$$
with $P_m^{(n)}/A_n\cong E^{(n)}_m, A_n/B_n^{\pm}\cong E^{(n\mp1)}_m, B_n^{\pm}/C_n\cong E^{(n\pm1)}_m$, and $C_n\cong E^{(n)}_m$ for $n>0$.  Using the notation above, define the following maps\footnote{The equations in (\ref{martinmaps}) are only defined when $m$ is sufficiently large.}:
\begin{equation}\label{martinmaps}\begin{array}{ll}
\alpha_0=\alpha_{0, m}:P^{(0)}_m\twoheadrightarrow P^{(0)}_m/B_0\cong C_1\hookrightarrow P^{(1)}_m\\
\beta_0=\beta_{0, m}:P^{(1)}_m\twoheadrightarrow P^{(1)}/B_1^+\cong P_m^{(0)}\\
\alpha_n=\alpha_{n,m}: P_m^{(n)}\twoheadrightarrow P_m^{(n)}/B^-_n\cong B^-_{n+1}\hookrightarrow P_m^{(n+1)} &  (n>0) \\
\beta_n=\beta_{n,m}:P_m^{(n+1)}\twoheadrightarrow P_m^{(n+1)}/B^+_{n+1}\cong B^+_n\hookrightarrow P_m^{(n)}  &  (n>0)\\
\gamma_n=\gamma_{n,m}: P_m^{(n)}\twoheadrightarrow P_m^{(n)}/A_n\cong C_n\hookrightarrow P_m^{(n)}  &  (n>0)
\end{array}
\end{equation}

We are now ready to give a general description of the nontrivial blocks in $\uRep(S_d; F)$. 

\begin{theorem}\label{blocksviamartin}  Suppose $d$ is a nonnegative integer and $\cat{B}$ is a nontrivial block in $\uRep(S_d; F)$.  Then $\cat{B}$ is equivalent as an additive category to the nontrivial block described in Theorem \ref{zeroblock}.
\end{theorem}

\begin{proof}  We may assume $d\not=0$.  Notice $$\Hom_{\uRep(S_d; F)}(L(\lambda^{(n)}), L(\lambda^{(n')}))=\Hom_{FP_m(d)}(P_m^{(n)}, P_m^{(n')})$$ whenever $m\geq |\lambda^{(n)}|, |\lambda^{(n')}|$.  Hence, by Proposition \ref{dimstuff}(2), it suffices to prove the maps defined by (\ref{martinmaps}) satisfy  equations (1)-(7) in Lemma \ref{calculations}.  Equations (1)-(5) are clearly satisfied. The fact that equations (6) and (7) are satisfied follows from the observation that  the maps $B^\mp_{n\pm1}\hookrightarrow P_m^{(n\pm1)}\twoheadrightarrow P_m^{(n\pm1)}/B^\pm_{n\pm1}\cong B^\pm_n$ factor through the maps $B^\mp_{n\pm1}\twoheadrightarrow B^\mp_{n\pm1}/C_{n\pm1}\cong C_n\hookrightarrow B^\pm_n$.  
\end{proof}

\begin{remark} Our proof of Theorem \ref{blocksviamartin} is a bit unsatisfactory since it is
based on rather deep results from \cite{MR1399030}. On the other hand Theorem \ref{blocksviamartin} follows from Theorem \ref{zeroblock}, Proposition \ref{blocksfixd}, along with Conjecture \ref{conjecture}.  Hence a proof of Conjecture \ref{conjecture} would complete a proof of Theorem \ref{blocksviamartin} which is independent from \cite{MR1399030}. 
\end{remark}


\appendix

\section{On lifting idempotents}\label{liftingappendix}    As we were unable to find an appropriate reference, in this appendix we give standard proofs of well-known results on lifting idempotents.
Assume $F$ is a field, $u$ is an indeterminate, and $A$ is a free $F[[u]]$-algebra of finite rank.  Finally, for $a\in A$, let $\bar{a}$ denote the image of $a$ under the quotient map $A\to A/Au$.  In this appendix we will show that any idempotent in the $F$-algebra $A/Au$ can be lifted to an idempotent in $A$.  First, we need the following lemma.

\begin{lemma}\label{invert}  $a$ is a unit in $A$ if and only if $\bar{a}$ is a unit in $A/Au$.
\end{lemma}

\begin{proof}  If $a$ is a unit in $A$, then clearly $\bar{a}$ is a unit in $A/Au$ with $\bar{a}^{-1}=\overline{a^{-1}}$.  Now suppose $\bar{a}$ is a unit in $A/Au$.  Then there exists $b\in A$ with $\bar{a}\bar{b}=\bar{1}$ in $A/Au$.  Hence $ab=1+xu$ for some $x\in A$.  Thus $ab\left(1-\sum_{i=1}^\infty x^iu^i\right)=1$ in $A$.  Thus $a$ is right invertible.  Similarly $a$ is left invertible, hence invertible.
\end{proof}

Now we are ready to prove the following theorem on lifting idempotents.

\begin{theorem}\label{appenlift}  Given an idempotent $e\in A/Au$, there exists an idempotent $\ep\in A$ such that $\bar{\ep}=e$ (we say $\ep$ is a lift of $e$).  Moreover, if two idempotents are conjugate in $A/Au$, then their lifts are conjugate in $A$.
\end{theorem}

\begin{proof} To prove the first statement let $A_i:=A/Au^i$ for $i>0$ and let $\psi_i$ denote the quotient map $A_i\to A_i/A_iu^{i-1}=A_{i-1}$ for each $i>1$.  Set $e_1=e\in A_1$.  We will inductively show that for each $i>1$ there exists an idempotent $e_i\in A_i$ with $\psi_i(e_i)=e_{i-1}$.  The first statement of the theorem will follow by letting $\ep$ be the unique element of $A$ such that $\ep\mapsto e_i$ under the quotient map $A\to A_i$ for all $i>0$.  Now, fix $i>1$ and assume $e_{i-1}$ is an idempotent in $A_{i-1}$.  Choose $a_i\in A_i$ with $\psi_i(a_i)=e_{i-1}$ and set $e_i=3a_i^2-2a^3_i$.  Then $\psi_i(e_i)=3e_{i-1}^2-2e_{i-1}^3=e_{i-1}$ in $A_{i-1}$.  It remains to show that $e_i$ is an idempotent in $A_i$.  Set $b_i=a_i^2-a_i$.  Then $\psi_i(b_i)=0$, hence $b_i\in A_iu^{i-1}$.  In particular, $b_i^2=0$ in $A_i$.  Moreover, $e_i=a_i-(2a_i-1)b_i$.  Since $a_i$ and $b_i$ clearly commute, $$e_i^2=a_i^2-2(2a_i-1)b_i=a_i^2-4b_i^2-2a_ib_i=a_i^2-2(a_i^3-a_i^2)=e_i\quad\text{in }A_i.$$

To prove the second statement of the theorem, suppose $\ep$ and $\ep'$ are idempotents in $A$.  First assume $\bar{\ep}=\bar{x}\overline{\ep'}\bar{x}^{-1}$ in $A/Au$ for some $x\in A$ with $\bar{x}$ invertible.  By Lemma \ref{invert}, $x$ is invertible.  
It follows that $x\ep'x^{-1}$ is an idempotent in $A$ with 
$\bar{\ep}=\overline{x\ep'x^{-1}}$ in $A/Au$.  Therefore, it suffices to show $\ep$ and $\ep'$ are conjugate in $A$ whenever $\bar{\ep}=\overline{\ep'}$.  To do so, consider the element $a:=\ep\ep'+(1-\ep)(1-\ep')\in A$.  If $\bar{\ep}=\overline{\ep'}$, then $\bar{a}=\bar{1}$.  Hence $a$ is a unit in $A$ (see Lemma \ref{invert}).  Moreover, $\ep a=\ep\ep'=a\ep'$ so that $\ep=a\ep'a^{-1}$.
\end{proof}


\renewcommand{\bibname}{\textsc{references}} 
\bibliographystyle{alpha}	
	\addcontentsline{toc}{chapter}{\bibname}
	\bibliography{references}

\end{document}